\documentclass{article}

\usepackage{microtype}
\usepackage{graphicx}
\usepackage{subfigure}
\usepackage{booktabs} % for professional tables

\usepackage[hyperindex,breaklinks]{hyperref}
\hypersetup{
	breaklinks=true,   % splits links across lines
	colorlinks=true,   % displays links as colored text instead of blocks
	pdfusetitle=true,  % \title and \author values into pdf metadata
}
\usepackage{xr}
\usepackage[utf8]{inputenc} % allow utf-8 input
\usepackage[T1]{fontenc}    % use 8-bit T1 fonts
\usepackage{ae} 
\usepackage{aecompl}
\usepackage{url}            % simple URL typesetting
\usepackage{booktabs}       % professional-quality tables
\usepackage{amsfonts}       % blackboard math symbols
\usepackage{nicefrac}       % compact symbols for 1/2, etc.
\usepackage{microtype}      % microtypography
\usepackage{amsmath,amsthm,amssymb}
\usepackage{bm}
\usepackage{graphicx}
\usepackage{caption}
\usepackage{enumitem}
\usepackage{algorithmic,algorithm}
\usepackage{natbib}

\newcommand{\X}{\mathcal{X}}
\newcommand{\bzt}{\boldsymbol{\zeta}}

\newcommand{\E}{\mathbb{E}}
\newcommand{\reals}{\mathbb{R}}

\newcommand{\argmin}{\mathop{\mathrm{arg\,min}{}}}
\newcommand{\argmax}{\mathop{\mathrm{arg\,max}{}}}

\newcommand{\ba}{{\mathbf a}}

\newcommand{\bx}{{\mathbf x}}
\newcommand{\bz}{{\mathbf z}}
\newcommand{\by}{{\mathbf y}}

\newcommand{\vertiii}[1]{{\left\vert\kern-0.25ex\left\vert\kern-0.25ex\left\vert #1
		\right\vert\kern-0.25ex\right\vert\kern-0.25ex\right\vert}}

\newtheorem{assumption}{Assumption}

\newtheorem{lemma}{Lemma}
\newtheorem{definition}{Definition}
\newtheorem{theorem}{Theorem}
\newtheorem{remark}{Remark}

\usepackage[accepted]{icml2020}

\icmltitlerunning{Quadratically Regularized Subgradient Methods for Weakly Convex Optimization with Weakly Convex Constraints}

\begin{document}

\twocolumn[
\icmltitle{Quadratically Regularized Subgradient Methods for Weakly Convex Optimization with Weakly Convex Constraints}



\icmlsetsymbol{equal}{*}

\begin{icmlauthorlist}
\icmlauthor{Runchao Ma}{1}
\icmlauthor{Qihang Lin}{1}
\icmlauthor{Tianbao Yang}{2}
\end{icmlauthorlist}

\icmlaffiliation{1}{Department of Business Analytics, University of Iowa, Iowa
City, IA, USA}
\icmlaffiliation{2}{Department of Computer Science, University
of Iowa, Iowa City, IA, USA}

\icmlcorrespondingauthor{Qihang Lin}{qihang-lin@uiowa.edu}

\icmlkeywords{Machine Learning, ICML}

\vskip 0.3in
]



\printAffiliationsAndNotice{}  % leave blank if no need to mention equal contribution

\begin{abstract}
		Optimization models with non-convex constraints arise in many tasks in machine learning, e.g., learning with fairness constraints or Neyman-Pearson classification with non-convex loss. Although many efficient methods have been developed with theoretical convergence guarantees for non-convex unconstrained problems, it remains a challenge to design provably efficient algorithms for problems with non-convex functional constraints. This paper proposes a class of subgradient methods for constrained optimization where the objective function and the constraint functions are weakly convex and nonsmooth. Our methods solve a sequence of strongly convex subproblems, where a quadratic regularization term is added to both the objective function and each constraint function. Each subproblem can be solved by various algorithms for strongly convex optimization. Under a uniform Slater’s condition, we establish the computation complexities of our methods for finding a nearly stationary point.
	\end{abstract}
	
	\section{Introduction}
	\label{sec:intro}
	Continuous optimization models with nonlinear constraints have been widely used in many disciplines including machine learning, statistics, and data mining with many real-world applications. A general optimization problem with inequality constraints is formulated as
	\begin{equation}
	\begin{aligned}
	\label{eq:gco}
	f^*\equiv \min_{\bx\in\mathcal{X}}\{f(\bx)&\equiv f_0(\bx)\}   \\
	 \text{s.t.} \quad g(\bx)&\equiv\max_{i=1,\dots,m}f_i(\bx)\leq 0
	\end{aligned}
	\end{equation}
	Here, we assume that $\mathcal{X}\subset\mathbb{R}^d$ is a compact convex set that allows for a simple projection and $f_i$ for $i=0,\dots,m$ are weakly-convex (potentially non-smooth) functions. A solution $\bar \bx\in\mathcal{X}$ is \emph{$\varepsilon$-optimal} if $f(\bar \bx)-f^*\leq \varepsilon$ and \emph{$\varepsilon$-feasible} if $\bar\bx\in\X$ and $g(\bar\bx)\leq \varepsilon$. Many optimization models in machine learning contain nonlinear constraints. Examples include Neyman-Pearson classification~\citep{rigollet2011neyman} and learning with dataset constraints~\citep{goh2016satisfying} (e.g. fairness constraints and churn rate constraints). 
	
	Optimization problems with a convex objective function and convex constraints have been well studied in literature with many efficient algorithms and their theoretical complexity developed~\citep{bertsekas2014constrained,bertsekas1999nonlinear,nocedal2006numerical}. However, the parallel development for optimization with non-convex objective functions and non-convex constraints, especially for theoretically provable algorithms, remains limited, restricting the practices of statistical modeling and decision making in many disciplines. It is well-known that finding a global minimizer for a general non-convex function without any constraints has been intractable~\cite{sahni1974computationally}. The difficulty will increase when constraints appear and will increase even further when those constraints are non-convex. 
	%In fact, even finding a feasible solution that satisfies all non-convex constraints is itself an intractable problem~\cite{mangasarian2017sufficient}. 
	
	Therefore, when designing an algorithm for~\eqref{eq:gco} with non-convex objective and constraint functions, the first question to be addressed is what kind of solutions can the algorithm guarantees and what complexity the algorithm has in order to find such solutions. In the recent studies on unconstrained or simply constrained\footnote{Here, being simply constrained means the feasible set is a simple set, e.g., a box or a ball, that allows for a closed-form for the projection mapping.} non-convex minimization~\citep{davis2018stochastic,Davis2018,davis2018complexity,davis2017proximally,drusvyatskiy2017proximal,Drusvyatskiy2018,DBLP:journals/siamjo/GhadimiL13a,DBLP:journals/mp/GhadimiL16,DBLP:journals/corr/abs/1805.05411,paquette2018catalyst,Reddi:2016:SVR:3045390.3045425,DBLP:conf/cdc/ReddiSPS16}, algorithms have been proposed to find a \emph{nearly stationary point}, which is a feasible solution close to another feasible solution where the subdifferential of the objective function almost contains zero. However, these methods and analysis cannot be applied to \eqref{eq:gco} as they require the exact projection to the feasible set which is hard to perform for \eqref{eq:gco} due to the functional constraints. To address this issue, in this work, we propose a class of first-order methods for~\eqref{eq:gco} where the objective function and the constraint functions are all weakly convex. Our methods solve a sequence of strongly convex subproblems where, different from the traditional proximal-point method, a quadratic regularization term is also added to each constraint function instead of just the objective function. Each subproblem can be solved by an algorithm for strongly convex optimization. Under a uniform Slater's condition, we establish the complexities of our methods for finding a nearly stationary point. We will discuss some applications of~\eqref{eq:gco} in machine learning next. %and summarize our main contributions at the end of this section. 

	%However, it is difficulty to obtain the solution with the same properties for constrained problem as it is challenging to guarantee an exactly feasible solution even if the constraints are convex. Motivated by the aforementioned developments, in this work, we consider finding a nearly $\epsilon$-stationary point,

	\subsection{Optimization Problems in Machine Learning with Nonlinear Constraints}\label{sec:apps}
	\textbf{Multi-class Neyman-Pearson Classification:} In multi-class classification, there exist $K$ classes of data, denoted by $\xi_k$ for $k=1,2,\dots,K$, each of which has its own distribution. To classify each data into one of the $K$ classes, one can rely on $K$ linear models $\bx_k$, $k=1,2,\dots,K$ and predict the class of a data point $\xi$ as 
	$\argmax_{k=1,2,\dots,K}\bx_k^\top \xi$. To achieve a high classification accuracy, we would like the value $\bx_k^\top \xi_k-\bx_l^\top \xi_k$ with $k\neq l$ to be positively large~\citep{weston1998multi,crammer2002learnability}, which can be achieved by minimizing the expected loss $\E\phi(\bx_k^\top \xi_k-\bx_l^\top \xi_k)$, where $\phi$ is a non-increasing  potentially non-convex loss function and $\E$ is the expectation taken over $\xi_k$. %Suppose misclassifying $\xi_k$ has a cost depending on $k$ regardless of the predicted class. 
	When  training these $K$ linear models, one can prioritize minimizing the loss on class $1$ while control the losses on other classes by solving 
	\begin{align*}
		\label{eq:NPclassification_multi}
		&\min_{\|\bx_k\|_2\leq\lambda,k=1,\dots,K} \sum_{l\neq 1}\E[\phi(\bx_1^\top \xi_1-\bx_l^\top \xi_1)] \\
		&{s.t.}~\sum_{l\neq k}\E[\phi(\bx_k^\top \xi_k-\bx_l^\top \xi_k)]\leq r_k\quad k=2,3,\dots,K,
	\end{align*}
	where $r_k$ controls the loss for class $k$ and $\lambda$ is a regularization parameter. 
	%If the misclassification cost depends on both the true class and the predicted class, we can further break each constraint in \eqref{eq:NPclassification_multi} according to the predicted classes and solve
	%\begin{eqnarray}
	%\label{eq:NPclassification_multi_multi}
	%\min_{\|\bx_k\|_2\leq\lambda} \E[\phi(\bx_1^\top \xi_1-\bx_1^\top \xi_2)],~\text{s.t.}~\E[\phi(\bx_k^\top \xi_k-\bx_l^\top \xi_k)]\leq r_k\quad k,l=1,2,\dots,K\text{ and }l\neq k.
	%\end{eqnarray}
	When $\xi$ follows the empirical distribution over a finite dataset, the expectations above are essentially sample averages so that this problem becomes a deterministic optimization problem. 
	
	\textbf{Learning Data-Driven Constraints:} Problem~\eqref{eq:gco} also covers many machine learning models with data-driven constraints~\citep{goh2016satisfying}. The examples include the constraints that impose conditions on the coverage rates, churn rates, or fairness of a predictive model. More details can be found in \citep{goh2016satisfying}. Here, we focus on learning a classifier with parity-based fairness constraints \cite{goh2016satisfying,zafar2015fairness,zafar2017parity}. Suppose $(\ba,b)$ is a point from a distribution $\mathcal{D}$ where $b\in\{1,-1\}$ is the label. Let $\mathcal{D}_M$ and $\mathcal{D}_F$ be two different distributions of points (not necessarily labeled), e.g., $\mathcal{D}_M$ and $\mathcal{D}_F$ may represent the male and female groups. The training of a classifier with fairness constraints can be formulated as 
%	\small
	\begin{align*}
		\min_{\|\bx\|_2\leq\lambda}&\E_{(\ba,b)\sim\mathcal{D}}[\phi(-b\ba^\top\bx)]\\
		{s.t.}~&\E_{\ba\sim\mathcal{D}_M}[\sigma(\ba^\top\bx)]+\beta\E_{\ba\sim\mathcal{D}_F}[\sigma(-\ba^\top\bx)]\leq r \\
		&\E_{\ba\sim\mathcal{D}_F}[\sigma(\ba^\top\bx)]+\beta\E_{\ba\sim\mathcal{D}_M}[\sigma(-\ba^\top\bx)]\leq r
	\end{align*}
%	\normalsize
	%$\sigma(z)=\max\{0,\min\{1,0.5+z\}\}$ or 
	where $\phi$ is a non-increasing potentially non-convex loss function, $\sigma=\frac{\exp(z)}{1+\exp(z)}$, $\lambda$ is a regularization parameter, $\beta$ is a positive balance parameter and $r$ is a constraint parameter. The objective function is the training loss of $\bx$.
	%, the minimization of which ensures a good classification accuracy of $\bx$. 
	The terms $\sigma(\ba^\top\bx)$ and $\sigma(-\ba^\top\bx)$ represent the predicted probabilities of $\ba$ being in the positive and the negative class, respectively. The left hand side of the first constraint will be large if the model $\bx$ is very ``unfair'' in the sense that it makes $\ba^\top\bx$ very negative for most of $\ba$ from $\mathcal{D}_M$ but very positive for most of $\ba$ from $\mathcal{D}_F$. The second constraint can be interpreted similarly.
	%Similarly, the left hand side of the second constraint will be large if the model $\bx$ makes $\ba^\top\bx$ very positive for most of $\ba$ from $\mathcal{D}_F$ but very positive for most of $\ba$ from $\mathcal{D}_M$. 
	Choosing  appropriate $r$ forces the left hand sides of both constraints low so that the obtained model will be fair to both  $\mathcal{D}_M$ and $\mathcal{D}_F$.
%	Another example of fairness constraint is given in Section~\ref{sec:exp}. 
	
	\subsection{Contributions} 
	\label{sec:contrib}
	We summarize our contributions as follows.
	\begin{itemize}
		\item We propose a class of algorithms (Algorithm~\ref{alg:iqrc}) for~\eqref{eq:gco} when all $f_i$ are weakly convex. This method approximately solves a strongly convex subproblem \eqref{eq:phit} in each main iteration with precision $O(\epsilon^2)$ using a suitable first-order method. We show that our method finds a nearly $\epsilon$-stationary point (Definition~\ref{def:stationary}) for \eqref{eq:gco} in $O(\frac{1}{\epsilon^2})$ main iterations. %(Theorem~\ref{thm:main}).  
		\item When each $f_i$ is a deterministic function, we develop a new variant of the switching subgradient method~\cite{swtichgradientpolyak} to solve \eqref{eq:phit}. We show that the complexity of Algorithm~\ref{alg:iqrc} for finding a nearly $\epsilon$-stationary point is $O(\frac{1}{\epsilon^4})$.		
		\item When each $f_i$ is given as an expectation of a stochastic function, we directly use the stochastic subgradient method by \cite{yu2017online} to solve \eqref{eq:phit}. We show that the complexity of Algorithm~\ref{alg:iqrc} for finding a nearly $\epsilon$-stationary point is $\tilde O(\frac{1}{\epsilon^6})$.\footnote{In this paper, $\tilde O(\cdot)$ suppresses all logarithmic factors of $\epsilon$.}
	\end{itemize}
	
	\section{Related Work}
	There has been growing interest in first-order algorithms for non-convex minimization problems with no constraints or simple constraints in both stochastic and deterministic setting. Initially, the research in this direction mainly focus on the problem with a smooth objective function \cite{DBLP:journals/siamjo/GhadimiL13a,yangnonconvexmo,DBLP:journals/mp/GhadimiL16,DBLP:conf/cdc/ReddiSPS16,Reddi:2016:SVR:3045390.3045425,DBLP:journals/corr/abs/1805.05411,DBLP:conf/icml/Allen-Zhu17,DBLP:conf/icml/ZhuH16,lacoste2016convergence}. Recently, more studies have been developed on the algorithms and theories for non-convex minimization problems with non-smooth objective functions after assuming the objective function is weakly  convex~\citep{Davis2018,davis2018stochastic,Drusvyatskiy2018,davis2017proximally,chen18stagewise,zhang2018convergence}. These works tackle the non-smoothness of objective function by introducing the Moreau envelope of objective function and analyze the complexity of finding a nearly stationary point. However, these methods are not directly applicable to \eqref{eq:gco} because of the functional constraints. 
	
	The studies on convex optimization with functional constraints have a long history~\citep[and references therein]{bertsekas2014constrained,bertsekas1999nonlinear,nocedal2006numerical,ruszczynski2006nonlinear}. The recent development in the first-order methods for convex optimization with convex constraints include~\cite{mahdavi-2012-stochstic,zhang2013logt,chen2016optimal,yang2017richer,wei2018solving,xu2018primal,xu2017global,xu2017first,yu2017online,lin2018levelsiam,lin2018levelfinitesum,bayandina2018mirror,fercoq2019almost} for deterministic constraints and \cite{Lan2016,yu2017simple} for stochastic constraints. \cite{wei2018primal} propose a primal-dual Frank-Wolfe method for \eqref{eq:gco} with non-convex $f_0$ but linear $f_i$ for $i=1,2,\dots,m$. Different from these works, this paper study the problems where the objective function and the constaints are all non-convex. \cite{sahin2019inexact} propose an inexact augmented Lagrangian method for \eqref{eq:gco} with non-convex $f_0$ and nonlinear equality constraints. A complexity $\tilde O(\frac{1}{\epsilon^3})$ is claimed in Corollary 4.2 in \cite{sahin2019inexact}, but there is an error in its proof. The authors claimed the complexity of solving their subproblem is $\tilde O(\frac{\lambda_{\beta_k}^2 \rho^2}{\epsilon_{k+1}})$ but it should be $\tilde O(\frac{\lambda_{\beta_k}^2 \rho^2}{\epsilon_{k+1}^2})$. (See \cite{sahin2019inexact} for the definitions of $\lambda_{\beta_k}, \rho$, and $\epsilon_{k+1}$). After correcting this error, following the same proof they used gives a total complexity of  $\tilde O(\frac{1}{\epsilon^4})$. \cite{nguyen2018dc} study the problem \eqref{eq:gco} with non-convex $f_i$ for $i=0,\dots,m$, but only from the perspective of optimality conditions. Optimization algorithms
    and convergence analysis are not considered in their work.
	
	%An exact penalty method has been developed by \cite{cartis2011evaluation} to solve non-convex optimization problems with non-convex constraints. Each subproblem in the exact penalty method was solved by either a first-order trust-region method or a quadratic regularization method. Their methods have been extended by \cite{wang2017penalty} to handle stochastic objective function and by \cite{cartis2014complexity} to remove the boundness assumption on the penalty parameters and obtain a lower complexity. However, these works assume the objective and constraint functions are all smooth while our results can be applied to non-smooth problems.
	
	We realize a paper by Boob et al.~\cite{boob2019proximal} was posted online simultaneously as our paper. The main algorithms (Algorithm 1 and 2 in \cite{boob2019proximal}) they proposed are similar to our Algorithm~\ref{alg:iqrc} in the sense that a similar subproblem~\eqref{eq:phit} is solved in each main iteration. The main difference between our paper and~\cite{boob2019proximal} is the assumptions made to ensure the boundness of the dual variables of subproblem~\eqref{eq:phit}, which is critical to the convergence analysis. The authors of~\cite{boob2019proximal} establish convergence result under various constraint qualification conditions including, Mangasarian-Fromovitz constraint qualification (MFCQ), strong MFCQ, and strong feasibility while we only consider a uniform Slater's condition (Assumption~\ref{assume:stochastic}B). Strong feasibility condition is stronger than our uniform Slater's condition but, on the other hand, is easier to verify. The relative strength between (strong) MFCQ and the uniform Slater's condition is unknown. In addition, we focus on the cases where the objective and constraint functions are either all deterministic or all stochastic while \cite{boob2019proximal} considers an additional case where only the objective is stochastic. In the stochastic case, we require the stochastic gradients to be bounded (Assumption~\ref{assume:stochasticnew}) while ~\cite{boob2019proximal} assume the boundness of the second moment of the stochastic gradients. The complexities of our methods and theirs for finding an $\epsilon$-nearly stationary point are the same in the dependency on $\epsilon$ in the dominating terms. Their complexity is more general in the sense that it involves non-dominating terms that depend on the smoothness parameters of the smooth components of the functions which we do not consider. 
	%As for implementation, the primal-dual method proposed by \cite{boob2019proximal} for solving subproblem \eqref{eq:phit} requires knowing an upper bound of the dual variables of~\eqref{eq:phit} and has to use a guess-and- check procedure when such a bound is not available while the method we use to solve \eqref{eq:phit} does not have this requirement.
	
	%They assume $f_i$ for $i=0,1,\dots,m$ in \eqref{eq:gco} can be written as $f_i=f_i^c+f_i^n$ where $f_i^c$ is convex but potentially non-smooth and $f_i^n$ is non-convex and continuously differentiable. This assumption is equivalent to our assumption that $f_i$ is weakly convex and potentially non-smooth. In fact, a $\rho$-weakly convex function $f_i$ can be written as $f_i=f_i^c+f_i^n$ with  $f_i^c=(f_i(\bx)+\frac{\rho}{2}\|\bx\|^2)$ and $f_i^n=-\frac{\rho}{2}\|\bx\|^2$ to satisfy their assumption. On the other hand, when $\nabla f_i^n$ is Lipschitz continuous, $f_i^n$ is weakly convex. So is $f_i^c+f_i^n$. 

	%\cite{cartis2010complexity}
	%\cite{cartis2012adaptive}

	\section{Preliminaries}\label{sec:pre}
	Let $\|\cdot\|$ be the $\ell_2$-norm. For $h:\reals^d\rightarrow \reals\cup\{+\infty\}$, the subdifferential of $h$ at $\bx$ is
%	\small
	\begin{align*}
	\partial h(\bx)=
	&\big\{\bzt\in\mathbb{R}^d\big| 
	h(\bx')
	\geq h(\bx)+\bzt^\top(\bx'-\bx) \\
	&+o(\|\bx'-\bx\|), ~\bx'\rightarrow\bx
	\big\},
	\end{align*}
%	\normalsize
	%\normalsize
	where $\bzt\in\partial h(\bx)$ is a subgradient of $h$ at $\bx$. 
	We say  $h$ is \emph{$\mu$-strongly convex ($\mu\geq0$)} on $\X$ if 
	$$
	h(\bx)\geq h(\bx')+\bzt^\top(\bx-\bx')+\frac{\mu }{2}\|\bx-\bx'\|^2
	$$
	for any $(\bx,\bx')\in\X\times\X$ and any $\bzt\in\partial h(\bx')$. We say $h$ is \emph{$\rho$-weakly convex} ($\rho\geq0$) on $\X$ if
	$$
	h(\bx)\geq h(\bx')+\bzt^\top(\bx-\bx')-\frac{\rho}{2}\|\bx-\bx'\|_2^2
	$$
	for any $(\bx,\bx')\in\X\times\X$ and any $\bzt\in\partial h(\bx')$. 
	We denote the normal cone of $\X$ at $\bx$ by $\mathcal{N}_\X(\bx)$ and the distance from $\bx$ to a set $S$ by $\text{Dist}(\bx,S)=\min_{\by\in S}\|\bx-\by\|$. 
	
	The following assumptions about \eqref{eq:gco} are made throughout the paper:
	\begin{assumption}
		\label{assume:stochastic}
		The following statements hold:
		\begin{itemize}
			\item[A.] $f_i(\bx)$ is closed and $\rho$-weakly convex with $\partial f_i(\bx)\neq \emptyset$ on any $\bx\in\mathcal{X}$ for $i=0,1,\dots,m$.
			\item[B.] $\min\limits_{\by\in\X} \lbrace g(\by) + \frac{\rho+\rho_\epsilon}{2}\|\by - \bx \|^2 \rbrace < -\sigma_\epsilon$ for any $\epsilon^2$-feasible solution $\bx$ ($\bx\in\X$ and $g(\bx)\leq \epsilon^2$) for some positive constants $\sigma_\epsilon$ and $\rho_\epsilon$. We call this condition \emph{uniform Slater's condition}.\footnote{The original Slater's condition states that $g(\bar\by)<0$ for some $\bar\by\in X$. Here, our assumption is stronger because it includes the term $\rho_\epsilon\|\bar\by - \bx \|^2$ and requires that inequality holds for any $\epsilon^2$-feasible solution $\bx$.}
			\item[C.] The domain $\X$ is compact such that $\max_{\bx,\bx'\in\X}\|\bx-\bx'\|\leq D$ for some constant $D$.
			\item[D.] $f_{\text{lb}}\equiv\min_{\bx\in\X}f(\bx)>-\infty$.
			\item[E.] We have access to an initial $\epsilon^2$-feasible solution $\bx_{\text{feas}}$ with $\bx_{\text{feas}}\in\X$ and $g(\bx_{\text{feas}})\leq \epsilon^2$.
			\item[F.] $\|\bzt\|\leq M$ for a constant $M$ for any $\bzt\in\partial f_i(\bx)$, $\bx\in\X$, and $i=0,\dots,m$.
		\end{itemize}
	\end{assumption}
	
	A function is $\rho$-weakly convex if it is differentiable and the gradient is $\rho$-Lipchitz continuous. Hence, the two applications given in Section~\ref{sec:apps} satisfy Assumption~\ref{assume:stochastic}A when the loss functions $\phi$ and $\sigma$ are smooth. It is easy to show that $g$ defined in \eqref{eq:gco} is also $\rho$-weakly convex under Assumption~\ref{assume:stochastic}A.
%	 An example satisfies Assumption~\ref{assume:stochastic}B will be provided in Section~\ref{sec:exp} (see Remark~\ref{eq:example}).
	 A discussion about Assumption~\ref{assume:stochastic}E is given in Remark~\ref{eq:initialfeasible}.

	Under Assumption~\ref{assume:stochastic}, \eqref{eq:gco} is a non-convex constrained optimization problem so that even finding an $\epsilon$-feasible solution is difficult in general, let alone a globally optimal solution. For a non-convex problem, one alternative goal is to find a \emph{stationary} point of \eqref{eq:gco}, i.e., a point $\bx_*\in\mathcal{X}$ that satisfies the following Karush-Kuhn-Tucker conditions (KKT) conditions \citep[Theorem 28.3]{rockafellar1970convex}
	\begin{equation}
	\label{eq:KKT}
	\begin{aligned}
	-\bzt_0^*-\sum_{i=1}^m\lambda_i^*\bzt_i^*\in\mathcal{N}_\X(\bx_*)&, 
	\quad \lambda_i^*f_i(\bx_*)=0, \\
	f_i(\bx_*)\leq0&,\quad\lambda_i^*\geq0,
	\end{aligned}
	\end{equation}
	where $\lambda_i^*$ is the Lagrangian multiplier corresponding to the constraint $f_i(\bx)\leq 0$ for $i=1,\dots,m$ and $\bzt_i^*\in\partial f_i(\bx^*)$ for $i=0,1,\dots,m$.  
	Since an exact stationary point is hard to find with a finite number of iterations by many algorithms, it is more common to aim at finding an \emph{$\epsilon$-stationary} point, i.e., a point $\widehat\bx\in\mathcal{X}$ satisfying
	\small
	\begin{equation}
	\label{eq:eKKT}
	\begin{aligned}
	\text{Dist}\bigg(-\widehat\bzt_0-\sum_{i=1}^m\widehat\lambda_i\widehat\bzt_i,\mathcal{N}_\X(\widehat\bx)\bigg)\leq \epsilon&, \quad
	 |\widehat\lambda_if_i(\widehat\bx)|\leq \epsilon, \\ f_i(\widehat\bx)\leq\epsilon&,\quad\widehat\lambda_i\geq0,
	 \end{aligned}
	\end{equation}
	\normalsize
	where $\widehat\lambda_i$ is a Lagrangian multiplier corresponding to the constraint $f_i(\widehat\bx)\leq 0$ for $i=1,\dots,m$ and $\widehat\bzt_i\in\partial f_i(\widehat\bx)$ for $i=0,1,\dots,m$.  However, there are two difficulties that prevent algorithms from finding an $\epsilon$-stationary: \textbf{(i) Non-smoothness:} When $f_0$ is non-smooth, computing an $\epsilon$-stationary point with finitely many iterations is challenging even if $f_0$ is convex and there is no constraint, e.g.,  $\min_{x\in\reals}|x|$, where $0$ is an exact stationary point while an algorithm may still return an $x\approx 0$ but $\neq0$ which is not $\epsilon$-stationary for any $\epsilon<1$. \textbf{(ii) Non-convex constraints:} When non-convex constraints appear, it is difficult to numerically find a point $\widehat\bx$ that satisfies the third inequality in \eqref{eq:eKKT}. With a highly infeasible $\widehat\bx$, the other two inequalities in \eqref{eq:eKKT} become less meaningful. 
	
	Therefore, to study \eqref{eq:eKKT} in a more tractable setting,  we follow~\cite{Davis2018,davis2017proximally,davis2018complexity,zhang2018convergence} to make the weak convexity assumption in Assumption~\ref{assume:stochastic}A and consider a function $\varphi_{\hat{\rho}}$ and a solution $\widehat\bx$ defined as 
%	\small
%	\vspace{-0.5em}
	\begin{equation}
	\begin{aligned}
	\label{eq:phi}
	\varphi_{\hat{\rho}} (\bx) \equiv \min_{\by\in\X} \big\{ &f(\by) + \frac{\hat{\rho}}{2}\|\by - \bx\|^2, \\
	&{s.t.}\quad g(\by) + \frac{\hat{\rho}}{2}\|\by - \bx\|^2\leq 0 \big\} ,\\
	\end{aligned}
	\end{equation}
%	\vspace{-0.5em}
	\begin{equation}
	\begin{aligned}
	\label{eq:phix}
	\widehat\bx \equiv \argmin_{\by\in\X} \big\{ &f(\by) + \frac{\hat{\rho}}{2}\|\by - \bx\|^2,\\ 
	&{s.t.}\quad g(\by) + \frac{\hat{\rho}}{2}\|\by - \bx\|^2\leq 0\big\},
	\end{aligned}
	\end{equation}
%	\normalsize
	where $\hat{\rho}\geq0$ is a \emph{regularization parameter}, $g$ and $f$ are defined as in \eqref{eq:gco}. 
	%Here, we omit the dependency of $\widehat\bx$ on $\hat{\rho}$ in the notation. 
	It is important to point out that $\varphi_{\hat{\rho}}$ is different from the Moreau envelope of the function\footnote{Here, $\mathbf{1}_{\X,g\leq 0}(\bx)$ denotes the indicator function of the feasible set $\{\bx\in\X|g(\bx)\leq 0\}$}  $f(\bx)+\mathbf{1}_{\X,g\leq 0}(\bx)$ which is defined as
	\small
	\begin{eqnarray}
	\label{eq:phi1}
	\tilde\varphi_{\hat{\rho}} (\bx) \equiv \min_{\by\in\X} \left\{ f(\by) + \frac{\hat{\rho}}{2}\|\by - \bx\|^2,\quad 
	\text{s.t.}\quad g(\by) \leq 0\right\}.
	\end{eqnarray}
	\normalsize
	The function $\tilde\varphi_{\hat{\rho}}$ was considered in~\cite{Davis2018,davis2017proximally,davis2018complexity,zhang2018convergence,rafique2018non} and their algorithm and analysis are based on the fact that \eqref{eq:phi1} is a convex minimization problem when there is no $g$, $f$ is $\rho$-weakly convex, and $\hat\rho\geq\rho$. However, for our problem \eqref{eq:gco} where $g$ exists and is $\rho$-weakly convex, \eqref{eq:phi1} is hard to evaluate even only approximately. Therefore, we include the term  $\frac{\hat{\rho}}{2}\|\by - \bx\|^2$ in the constraint of \eqref{eq:phi} and \eqref{eq:phix} so that the minimization problem has a $(\hat\rho-\rho)$-strongly convex objective function and $(\hat\rho-\rho)$-strongly convex constraints when $\hat\rho\geq\rho$. As a result of strong convexity, the solution $\widehat\bx$ defined in \eqref{eq:phix} is unique and can be closely approximated by solving \eqref{eq:phi} or \eqref{eq:phix}.
	
	As an extension to the findings in~\cite{Davis2018,davis2017proximally,davis2018complexity,zhang2018convergence,rafique2018non}, the quantity $\|\bx-\widehat\bx\|$ with $\widehat\bx$ defined in \eqref{eq:phix} can be used as a measure of the quality of a solution $\bx$. More specifically, let $\widehat\lambda$ be the Lagrangian multiplier that satisfies the following KKT conditions together with $\widehat\bx$ in \eqref{eq:phix}:
%	\small
	\begin{equation}
	\begin{aligned}
	\label{eq:KKTprox}
	-\widehat\bzt_0  -\hat{\rho}(\widehat\bx - \bx)-\widehat\lambda \left(\widehat\bzt+\hat{\rho}(\widehat\bx - \bx)\right)\in\mathcal{N}_\X(\widehat\bx),\\
	\quad \widehat\lambda \left(g(\widehat\bx)+ \frac{\hat{\rho}}{2}\|\widehat\bx - \bx\|^2\right)=0,\\
	g(\widehat\bx)+ \frac{\hat{\rho}}{2}\|\widehat\bx - \bx\|^2\leq0, \\
	 \widehat\lambda\geq0
	\end{aligned}
	\end{equation}
%	\normalsize
	where $\widehat\bzt_0\in\partial f_0(\widehat\bx)$ and $\widehat\bzt\in\partial g(\widehat\bx)$.	These conditions imply
%	\small
	\begin{align*}
	    	\text{Dist}(-\widehat\bzt_0-\widehat\lambda \widehat\bzt, \mathcal{N}_\X(\widehat\bx) )\leq(1+\widehat\lambda)\hat{\rho}\|\widehat\bx-\bx\|, \\
		 |\widehat\lambda g(\widehat\bx)|=\frac{\widehat\lambda\hat{\rho}}{2}\|\widehat\bx - \bx\|^2,\\ g(\widehat\bx)\leq0, \\
		  \widehat\lambda\geq0.
	\end{align*}
%	\normalsize
	Therefore, in the scenario where $\|\widehat\bx - \bx\|\leq\epsilon$, $\hat{\rho}=O(1)$ and $\widehat\lambda=O(1)$, we have $	\text{Dist}(-\widehat\bzt_0-\widehat\lambda \widehat\bzt, \mathcal{N}_\X(\widehat\bx) )=O(\epsilon)$, $|\widehat\lambda g(\widehat\bx)|=O(\epsilon^2)$, and $g(\widehat\bx)\leq0$, which means $\widehat\bx$ is feasible and satisfies the optimality conditions of the original problem \eqref{eq:gco} with $O(\epsilon)$ precision and $\bx$ is only $\epsilon$-away from $\widehat\bx$. With this property, we can say $\bx$ is near to an $\epsilon$-stationary point (i.e., $\widehat\bx$) of  \eqref{eq:gco}.  In Lemma~\ref{thm:boundlambda} below, we will show that $\widehat\lambda=O(1)$ when $\hat\rho\in(\rho,\rho+\rho_\epsilon]$ and $\bx$ is $\epsilon^2$-feasible under Assumption~\ref{assume:stochastic}B. We formally define the solution we want to compute as follows.
	%We formalize this result as follows.
	
	\iffalse
	\begin{lemma}
		A solution $\bx$ is $\epsilon$-close to an $\epsilon$-stationary point of \eqref{eq:gco} if $(1 + \widehat\lambda )\hat{\rho} \|\bx -\widehat{\bx}\| \leq \epsilon$.
	\end{lemma}
	\begin{proof}
		$0 \in \partial f(\by) + \hat{\rho}(\by-\bx) + \lambda_\bx(\partial g(\by) + \hat{\rho}(\by-\bx)) + N_{\X}(\by)$. Since $\hat{\bx}$ is a minimizer of the problem, we have 
		\begin{align*}
		\hat{\bx} \in \lbrace \by | \text{dist}(0,\partial f(\by) + \lambda_\bx \partial g(\by) + N_{\X}(\by))^2 \leq (\hat{\rho} + \lambda_\bx \hat{\rho})^2 \| \bx - \hat{\bx}\|^2 \rbrace \rbrace
		\end{align*}
		So, if $\max \lbrace \| \bx - \hat{\bx}\|^2, (\hat{\rho} + \lambda_\bx \hat{\rho})^2 \|\bx - \hat{\bx}\|^2 \rbrace \leq \epsilon$, we have point which is $\epsilon$-close to an $\epsilon$-stationary point. Since we choose $\hat{\rho}  > \max \lbrace \rho, 1 \rbrace$, $(\hat{\rho} + \lambda_\bx \hat{\rho})^2 \|\bx - \hat{\bx}\|^2 \geq \|\bx - \hat{\bx}\|^2$. Then our goal in theoretical analysis is to prove $(\hat{\rho} + \lambda_\bx \hat{\rho})^2 \|\bx - \hat{\bx}\|^2 \leq \epsilon$.
	\end{proof}
	\fi

	\begin{definition}
		\label{def:stationary}
		A point $\bx\in\X$ is called a \textbf{nearly $\epsilon$-stationary point} of \eqref{eq:gco} if $\|\widehat\bx - \bx\|\leq\epsilon$ where $\widehat\bx$ is defined in \eqref{eq:phix} with respect to $\bx$ and $\hat\rho$. 
		%and $\widehat\lambda$ is the corresponding Lagrangian multiplier..
	\end{definition}
	Next, we propose a numerical method for finding a nearly $\epsilon$-stationary point of \eqref{eq:gco} with theoretical complexity analysis. The proofs for all theoretical results are given in the supplementary file.
	
	\section{Inexact Quadratically Regularized Constrained Method}
	%In general, optimizing a non-convex and non-smooth
	%\begin{definition}
	%	A function $h: \reals^d \rightarrow \reals\cup\{+\infty\}$ is $\rho$-weakly convex if $h(\cdot) + \frac{\rho}{2}\|\cdot\|^2$ is convex.
	%\end{definition}
	%\begin{proposition}
	%	If $h: \reals^d \rightarrow \reals$ is differentiable with a $L$-Lipschtiz continuous gradient, $h$ is $\rho$-weakly convex with $L=\rho$.
	%\end{proposition}
	
	The method we proposed is motivated by the recent studies on the inexact proximal methods by \cite{davis2017proximally,rafique2018non,lin2018solving} which originates from the proximal point method~\cite{citeulike:9472207}. The authors of \cite{davis2017proximally} consider $\min_{\bx\in\X}f(\bx)$ with a $\rho$-weakly convex and non-smooth $f(\bx)$. In their approach, given the iterate $\bx_t\in\mathcal{X}$, they generate the next iterate $\bx_{t+1}$ by approximately solving the following convex  subproblem  
%	\small
	\begin{equation}\label{subminproblem}
	\bx_{t+1}\approx\argmin_{\by \in \X} f(\by)+\frac{\hat\rho}{2}\|\by-\bx_t\|_2^2
	\end{equation}
%	\normalsize
	using the standard stochastic subgradient (SSG) method. Then, $\bx_{t+1}$ will be used to construct the next subproblem in a similar way. Similar approaches have been developed for solving non-convex non-concave min-max problems by \cite{rafique2018non,lin2018solving}.
	%They show that their method finds a nearly $\epsilon$-stationary point within a time complexity of $\tilde O(\frac{1}{\epsilon^4})$.
	
	Similar to their approaches, we will generate $\bx_{t+1}$ from $\bx_t$ by approximately solving 
	%\eqref{eq:phi} with $\bx=\bar\bx_t$, which can be stated as
	\small
	\begin{equation}
	\begin{aligned}
	\label{eq:phit}
	\bx_{t+1} \approx\widehat\bx_t\equiv \argmin_{\by\in\X} \big\{ &f(\by) + \frac{\hat{\rho}}{2}\|\by - \bx_t\|^2, \\ 
	&{s.t.}\quad g(\by) + \frac{\hat{\rho}}{2}\|\by - \bx_t\|^2\leq 0\big\}.
	\end{aligned}
	\end{equation}
	\normalsize
	However, the SSG method cannot be directly applied to \eqref{eq:phit} due to the constraints $g(\by) + \frac{\hat{\rho}}{2}\|\by - \bx_t\|^2\leq 0$. Thanks to the recent development in the first-order methods for nonlinear constrained convex optimization, there are existing techniques that can potentially be used as a subroutine to solve \eqref{eq:phit} in our main algorithm. To facilitate the description of our main algorithm and its anlaysis, we formally define the subroutine with the property we need as follows.
	%an $\epsilon$-feasible and $\epsilon$-optimal solution for \eqref{eq:phi} defined with $\bx$, namely,
	\begin{definition}
		\label{def:subroutine}
		An algorithm $\mathcal{A}$ is called an \textbf{oracle for \eqref{eq:phit}} if, for any $t\geq0$, $\hat\rho>0$, $\hat\epsilon>0$, $\delta\in(0,1)$, and $\bx_t\in\X$, it finds (potentially stochastic)\footnote{Here, we allow $\mathcal{A}$ to be a stochastic algorithm.} $\bx_{t+1}\in\X$ such that, with a probability of at least $1-\delta$,
%		\small
		\begin{align*}
		    f(\bx_{t+1}) + \frac{\hat{\rho}}{2}\|\bx_{t+1} - \bx_t\|^2-f(\widehat\bx_t) -\frac{\hat{\rho}}{2}\|\widehat\bx_t - \bx_t\|^2&\leq\hat\epsilon^2, \\
		    g(\bx_{t+1}) + \frac{\hat{\rho}}{2}\|\bx_{t+1} - \bx_t\|^2&\leq \hat\epsilon^2
		\end{align*}
%		\normalsize
		where $\widehat\bx_t$ is defined in \eqref{eq:phit}. We denote the output of $\mathcal{A}$ by $\bx_{t+1}=\mathcal{A}(\bx,\hat\rho,\hat\epsilon,\delta)$.
	\end{definition}

	Before we discuss which algorithms to use as the orcale, we first present the main algorithm, the inexact quadratically regularized constrained (IQRC) method, in Algorithm \ref{alg:iqrc} and analyze the number of iterations it needs for finding a nearly $\epsilon$-stationary point.
	
	\begin{algorithm}[tb]
		\caption{Inexact Quadratically Regularized Constrained (IQRC) Method}
		\label{alg:iqrc}
		\begin{algorithmic}[1]
			\STATE {\bfseries Input:} An $\epsilon^2$-feasible solution $\bx_0=\bx_{\text{feas}}$ (Assumption~\ref{assume:stochastic}E), $\hat\rho>\rho$, $\delta\in(0,1)$, $\hat\epsilon=\min\Big\{1,\sqrt{\frac{\hat\rho-\rho}{4}}\Big(\frac{M+\hat\rho D }{\sqrt{2 \sigma_{\epsilon} (\hat{\rho} - \rho)}}+1\Big)^{-\frac{1}{2}}\Big\}\epsilon $, the number of iterations $T$, and an oracle $\mathcal{A}$ for \eqref{eq:phit}.
			\FOR{$t=0,\dots,T-1$}
			\STATE $\bx_{t+1}=\mathcal{A}(\bx_t,\hat\rho,\hat\epsilon,\frac{\delta}{T})$
			%\STATE Find an $\epsilon$-feasible and $\epsilon$-optimal solution $\bx_{t+1}$ for \eqref{eq:phi} for $\bx=\bx_t$. 
			\ENDFOR
			\STATE {\bfseries Output:} $\bx_R$ where $R$ is a random index uniformly sampled from $\lbrace 0,\dots,T \rbrace$.
		\end{algorithmic}
	\end{algorithm}
	
	The following lemma shows that the optimal Lagrangian multiplier of \eqref{eq:phit} is uniformly bounded for all $t$ under Assumption~\ref{assume:stochastic}. This is critical for establishing the convergence of Algorithm~\ref{alg:iqrc}.
	\begin{lemma}
		\label{thm:boundlambda}
		Suppose $\hat\rho\in(\rho,\rho+\rho_\epsilon]$. Let $\bx_t$ be generated by Algorithm~\ref{alg:iqrc}, $\widehat\bx_t$ be defined in \eqref{eq:phit}, and $\lambda_t$ be the Lagrangian multiplier in the KKT conditions \eqref{eq:KKTprox} of \eqref{eq:phit} satisfied by $\widehat\bx_t$. We have $\lambda_t \leq \frac{M+\hat\rho D }{\sqrt{2 \sigma_\epsilon (\hat{\rho} - \rho)}}$ for $t=0,1,2,\dots,T-1$ with a probability of at least $1-\delta$.
	\end{lemma}

	\begin{theorem}
		\label{thm:main}
		Under Assumption~\ref{assume:stochastic}, Algorithm~\ref{alg:iqrc} guarantees $	\mathbb{E}_R\| \bx_R - \widehat\bx_R \|^2\leq\epsilon^2$ with a probability of at least $1-\delta$ if $T \geq  \frac{4(f(\bx_0) - f_{\text{lb}})}{\epsilon^2(\hat\rho-\rho)}$, where the expectation is taken over $R$.
		%		\begin{align*}
		%	\mathbb{E} \| \bx_R - \widehat\bx_R \|^2\leq \frac{4M \epsilon^2}{\sqrt{2 \sigma (\hat{\rho} - \rho)}(\hat{\rho} - \rho)}+\frac{6\epsilon^2}{(\hat{\rho}-\rho)}
		%	\end{align*}
	\end{theorem}
	\begin{remark}
	\label{eq:initialfeasible}
	Algorithm~\ref{alg:iqrc} requires the access to an $\epsilon^2$-feasible solution $\bx_0=\bx_{\text{feas}}$ (Assumption~\ref{assume:stochastic}E).
	When solving \eqref{eq:gco} without an initial feasible solution, a typical guarantee of an algorithm (e.g. \citep{cartis2011evaluation,cartis2014complexity}) is that it either finds an $\epsilon$-feasible and $\epsilon$-stationary point of \eqref{eq:gco} or finds a point which is an $\epsilon$-stationary point of $g$ but infeasible to \eqref{eq:gco}. In the later case, the solution is typically trapped in a local minimum of $g$ where $g(\bx)$ is not small, which can happen due to non-convexity of $g$. Therefore, when $\bx_{\text{feas}}$ is not available, our method will have such type of guarantee as long as a subgadient method (e.g. \cite{Davis2018}) is first applied to $\min_{\bx \in \mathcal{X}} g(\bx)$ with $\mathcal{O}(\frac{1}{\epsilon^4})$ iterations, which will return a nearly $\epsilon$-stationary point of $g$, denoted by $\bx_{\text{temp}}$. Then if  $g(\bx_{\text{temp}})\leq\epsilon^2$, we start Algorithm~\ref{alg:iqrc} with $\bx_0=\bx_{\text{temp}}$. If not, we are in the second case mentioned above, namely, we have found a nearly $\epsilon$-stationary point of $g$ which is infeasible to \eqref{eq:gco}, and $\bx_{\text{temp}}$ is returned as the final output. Adding this step to our method does not change the order of magnitude of its complexity. 
	\end{remark}
	
	According to Theorem~\ref{thm:main}, in order to find an nearly $\epsilon$-stationary point in expectation, we have to call the oracle $\mathcal{A}$ $O(1/\epsilon^2)$ times. Therefore, the totally complexity of Algorithm~\ref{alg:iqrc} highly depends on the complexity of $\mathcal{A}$ for a given $\epsilon$. In the next sections, we will discuss the methods that can be used as  $\mathcal{A}$ when $f_i$ have different properties.
	
	\subsection{Oracle for Deterministic Problem}
	In this section, we assume that we can calculate any $\bzt\in\partial f_i(\bx)$ for any $\bx\in\X$.
	%In addition to Assumption~\ref{assume:stochastic}, we make the following assumption.
	%\begin{assumption}
	%	\label{assume:determinisitic}
	%	There exists a constant $M$ such that $\|\bzt\|\leq M$ for any $\bzt\in\partial f_i(\bx)$ for any $\bx\in\X$ for for $i=0,1,\dots,m$.
	%\end{assumption}
	%We consider a possible oracle for solving \eqref{eq:phit} at the $t$th iteration of Algorithm~\ref{alg:iqrc}. 
	We define 
% 	\small
% 	\begin{equation}
% 	\begin{aligned}
% 	\label{eq:FandG}
% 	F(\bx):= f(\bx) + \frac{\hat{\rho}}{2}\|\bx - \bx_t \|^2,\\
% 	G(\bx):= g(\bx) + \frac{\hat{\rho}}{2}\|\bx - \bx_t \|^2
% 	\end{aligned}
% 	\end{equation}  
% 	\normalsize
%    \small
	\begin{equation}
	\begin{aligned}
	\label{eq:FandG}
	F(\bx):= f(\bx) + \frac{\hat{\rho}}{2}\|\bx - \bx_t \|^2,\\
	G(\bx):= g(\bx) + \frac{\hat{\rho}}{2}\|\bx - \bx_t \|^2
	\end{aligned}
	\end{equation}  
%	\normalsize
	so that problem~\eqref{eq:phit} becomes $\min_{\bx\in\X}F(\bx)\text{ s.t. }G(\bx)\leq0$. We define $F'(\bx)$ and $G'(\bx)$ as any subgradient of $F$ and $G$, respectively. Under Assumption~\ref{assume:stochastic}C and Assumption~\ref{assume:stochastic}F, we have $\|F'(\bx)\|\leq M+\hat{\rho}D$ and $\|G'(\bx)\|\leq M+\hat{\rho}D$ for any $\bx\in\X$.
	
	Because problem~\eqref{eq:phit} is non-smooth, we consider the Polyak’s switching subgradient method~\cite{swtichgradientpolyak}, which is also analyzed in~\cite{nesterov2013introductory} and recently extended by~\cite{bayandina2018mirror,Lan2016}. The method we propose here is a new variant of that method for a strongly convex problem. The details are given in Algorithm~\ref{alg:MD} where $\text{Proj}_\X(\bx)$ represents the projection of $\bx$ to $\X$.
	%and $F'(\bx)$ and $G'(\bx)$ represent any subgradient in the subdifferential $\partial F(\bx)$ and $\partial G(\bx)$, respectively.
	Different from~\cite{bayandina2018mirror}, our Algorithm~\ref{alg:MD} only uses a single loop instead of double loops. It is also different from~\cite{Lan2016} in the sense that our method keeps every intermediate solution $\epsilon$-feasible for \eqref{eq:phi} while the method in  \cite{Lan2016}  only ensures  $\epsilon$-feasibility after a fixed number of iterations. Moreover, \cite{Lan2016} requires knowing the total number of iterations before hand in order to design the step size $\gamma_k$ while our method does not. 
	
	\begin{algorithm}[tb]
				\caption{Switching subgradient method for the subproblem \eqref{eq:phit}}
				\label{alg:MD}
				\begin{algorithmic}[1]
					\STATE {\bfseries Input:} $\bz_0=\bx_t\in \X$, $\hat\rho>\rho$ and $\hat\epsilon>0$.
					\STATE Set $I=\emptyset$ and $F$ and $G$ as in \eqref{eq:FandG}.
					\STATE Set \small$K =\left\lceil \frac{4(M^2+\hat{\rho}D^2)}{(\hat\rho-\rho)\hat\epsilon^2}\right\rceil$\normalsize
					\FOR{$k=0,\dots,K-1$}
					\STATE $\gamma_k = \frac{2}{(\hat\rho-\rho)(k+2)}$
					\IF{$G(\bz_k) \leq \hat\epsilon^2$}
					\STATE $I\leftarrow I\cup\{k\}$.
					\STATE $\bz_{k+1} = \text{Proj}_\X (\bz_k - \gamma_k F'(\bz_k))$
					\ELSE
					\STATE $\bz_{k+1} = \text{Proj}_\X (\bz_k - \gamma_k G'(\bz_k))$
					\ENDIF
					\ENDFOR
					\STATE {\bfseries Output:} $\bx_{t+1} = \frac{\sum_{k \in I} (k+1)\bz_k}{\sum_{k \in I} (k+1)}$.
				\end{algorithmic}
			\end{algorithm}

	The convergence of  Algorithm~\ref{alg:MD} is given below whose proof follows the idea of Section 3.2 in~\cite{lacoste2012simpler}. However, the original analysis in \cite{lacoste2012simpler} is for the subgradient method applied to unconstrained problems while our analysis is for the switching subgradient method applied to constrained problems.
	\begin{theorem}\label{thm:mdconverge}
		%Suppose $K \geq \frac{4(M^2+\hat{\rho}D^2)}{\mu \epsilon}$ and  $\gamma_k = \frac{2}{\mu(k+2)}$ in Algorithm~\ref{alg:MD}, where $\mu := \hat{\rho} - \rho $.
		Under Assumption~\ref{assume:stochastic}, Algorithm~\ref{alg:MD} guarantees $F(\bx_{t+1}) - F(\widehat\bx_t)\leq \hat\epsilon^2$ and $G(\bx_{t+1}) \leq \hat\epsilon^2$ deterministically and can be used as an oracle $\mathcal{A}$ for \eqref{eq:phit}. The complexity of Algorithm~\ref{alg:iqrc} using Algorithm~\ref{alg:MD} as an oracle is therefore $O(\frac{1}{\epsilon^4})$.%~\footnote{Algorithm~\ref{alg:MD} is deterministic so its complexity above does not depend on $\delta$.}
		%(The exact expression of the complexity is given in the proof.)
		%where $\widehat\bx_t$ is the optimal solution of \eqref{eq:phi} with $\bx=\bx_t$.
	\end{theorem}
	%By this theorem and the value of $K$ in Algorithm~\ref{alg:MD}, 
	%\begin{corollary}\label{thm:mdconverge_cor}
	%	Suppose $K \geq \frac{4(M^2+\hat{\rho}D^2)}{\mu \epsilon}$ and  $\gamma_k = \frac{2}{\mu(k+2)}$ in Algorithm~\ref{alg:MD}, where $\mu := \hat{\rho} - \rho $. Algorithm~\ref{alg:MD} guarantees
	%	$F(\bar\bx) - F(\bar\bx_t)\leq \epsilon$ and $G(\bar\bx) \leq \epsilon$. 
	%	%where $\widehat\bx_t$ is the optimal solution of \eqref{eq:phi} with $\bx=\bx_t$.
	%\end{corollary}
	\begin{remark}
	\label{eq:smoothcase}
	Although the main focus of this paper is the case when $f_i$ is non-smooth in \eqref{eq:gco} for $i=0,\dots,m$, our results can be easily extended to the case where each $f_i$ is differentiable with an $L$-Lipschitz continuous gradient. In this case, the subproblem \eqref{eq:phit} is written as computing
	\small
	\begin{equation*}
	\begin{aligned}
	\label{eq:phitsmooth}
	\bx_{t+1} \approx\widehat\bx_t\equiv &\argmin_{\by\in\X} \big\{ f_0(\by) + \frac{\hat{\rho}}{2}\|\by - \bx_t\|^2,\\ 
	&{s.t.}\quad f_i(\by) + \frac{\hat{\rho}}{2}\|\by - \bx_t\|^2\leq 0,~i=1,\dots,m\big\}.
	\end{aligned}
	\end{equation*}
	\normalsize
	Since the objective function and constraint functions here are all strongly convex and smooth, there exist some algorithms that can be used as an oracle for \eqref{eq:phit} satisfying Definition~\eqref{def:subroutine}. The examples include the level-set method~\cite{lin2018levelsiam} and the augmented Lagrangian method~\cite{xu2017global} whose complexity for computing  $\bx_{t+1}=\mathcal{A}(\bx,\hat\rho,\hat\epsilon,\delta)$ is $O(\frac{1}{\hat\epsilon})$. Since $\hat\epsilon=O(\epsilon)$, the complexity of Algorithm~\ref{alg:iqrc} using \cite{lin2018levelsiam}  or \cite{xu2017global} as the oracle is $O(\frac{1}{\epsilon^2})\times O(\frac{1}{\hat\epsilon})=O(\frac{1}{\epsilon^3})$. 
	\end{remark}
	
	\subsection{Oracle for Stochastic Problem}
	\vspace{-2mm}
	In this section, we consider the scenario where only a stochastic unbiased estimation for the subgradient of $f_i$ is available. In addition to Assumption~\ref{assume:stochastic}, we make the following assumption.
	\begin{assumption}
		\label{assume:stochasticnew}
		For any $\bx\in\X$ and any $i=0,1\dots,m$, we can compute a stochastic estimation $\theta_i(\bx)$ and a stochastic gradient $\bzt_i(\bx)$ of $f_i$ such that $\mathbb{E}\theta_i(\bx)=f_i(\bx)$ and $\mathbb{E}\bzt_i(\bx)\in\partial f_i(\bx)$. Moreover, there exist constants $M_0$ and $M_1$ such that $\|(\theta_1(\bx),\theta_2(\bx),\dots,\theta_m(\bx))\|\leq M_0$ and 
		$\|\bzt_i(\bx)\|\leq M_1$ for any $\bx$ almost surely.
	\end{assumption}
	A typical situation where this assumption holds is the stochastic optimization where $f_i\equiv \mathbb{E}F_i(\bx,\xi)$ and $\xi$ is a random variable. In that case, we can sample $\xi$ and compute $\theta_i(\bx)=F_i(\bx,\xi)$ and compute $\bzt_i(\bx)$ as a subgradient of $F_i(\bx,\xi)$ with respect to $\bx$.
	
	Under this setting, when solving the subproblem \eqref{eq:phit}, it is not possible to construct an unbiased stochastic subgradient for $g$ in \eqref{eq:gco} or $G$ in \eqref{eq:FandG} due to the maximization operator in their definitions. Hence, we treat each $f_i$ as an individual function and define 
%	\small
	\begin{equation*}
	    F_i(\bx):= f_i(\bx) + \frac{\hat{\rho}}{2}\|\bx - \bx_t \|^2\text{ for }i=0,1,\dots,m
	\end{equation*}
%	\normalsize
	so that problem~\eqref{eq:phit} becomes $\min_{\bx\in\X}F(\bx)\text{ s.t. }F_i(\bx)\leq0$ for $i=0,1,\dots,m$.
	Note that $F_i$ is still $(\hat\rho-\rho)$-strongly convex. Its stochastic estimation is $\theta_i(\bx)+\frac{\hat{\rho}}{2}\|\bx - \bx_t \|^2$ which satisfies 
	$\|(\theta_i(\bx)+\frac{\hat{\rho}}{2}\|\bx - \bx_t \|^2)_{i=1}^m\|\leq\|(\theta_i(\bx))_{i=1}^m\| +\frac{\hat{\rho}\sqrt{m}}{2}\|\bx - \bx_t \|^2\leq M_0+\frac{\hat{\rho}\sqrt{m}D^2}{2}\equiv \tilde M_0$. Its stochastic gradient is $\bzt_i(\bx)+\hat\rho(\bx-\bx_t)$ which satisfies $\|\bzt_i(\bx)+\hat\rho(\bx-\bx_t)\|\leq M_1+\hat\rho D\equiv \tilde M_1$.

	The switching subgradient method (Algorithm~\ref{alg:MD}) and its variants~\cite{bayandina2018mirror,Lan2016} cannot handle stochastic constraints functions unless a large high-cost mini-batch is used per iteration~\cite{Lan2016}. Therefore, we consider using the online stochatsic subgradient method by~\cite{yu2017online} which allows for both stochastic objective function and stochastic constraints. We present their method in Algorithm~\ref{alg:Yu} and analyze the complexity of Algorithm~\ref{alg:iqrc} when using their method as the oracle.
	
	\begin{algorithm}[tb]
				\caption{Online stochastic subgradient method by \cite{yu2017online} for the subproblem \eqref{eq:phit}}
				\label{alg:Yu}
				\begin{algorithmic}[1]
					\STATE {\bfseries Input:} $\bz_0=\bx_t\in \X$, $\hat\rho>\rho$, $\hat\epsilon>0$, and the number of iterations $K$.
					\STATE Set  $V=\sqrt{K}$ and $\alpha=K$.
					\STATE Set $Q_0^i=0$ for $i=1,\dots,m$.
					\FOR{$k=0,\dots,K-1$}
					\STATE $\tilde\theta_i^k=\theta_i(\bz_k)+\frac{\hat\rho}{2}\|\bz_k-\bx_t\|^2$ and $\tilde\bzt_i^k=\bzt_i(\bz_k)+\hat\rho(\bz_k-\bx_t)$ for $i=0,1,2,\dots,m$.
					\STATE 
					\small
					$
					\bz_{k+1}=\argmin\limits_{\bz\in \X}\left\{
					\begin{array}{l}
					\big(V\bzt_0^k+\sum\limits_{i=1}^mQ_k^i\bzt_i^k\big)^\top (\bz-\bz_k)\\
					+\alpha\|\bz-\bz_k\|^2
					\end{array}
					\right\}
					$
					\normalsize
					\STATE \small$Q_{k+1}^i=\max\{Q_k^i+\tilde\theta_i^k+(\tilde\bzt_i^k)^\top(\bz_{k+1}-\bz_k),0\}$ $i=1,2,\dots,m$\normalsize.
					\ENDFOR
					\STATE {\bfseries Output:} $\bx_{t+1} = \frac{1}{K}\sum_{k=0}^{K-1}\bz_k$.
				\end{algorithmic}
			\end{algorithm}

	\begin{theorem}%\cite[Theorem 2]{yu2017online}
		\label{thm:Yu}
		Under Assumption~\ref{assume:stochastic} and \ref{assume:stochasticnew}, Algorithm~\ref{alg:Yu} guarantees $F(\bx_{t+1}) - F(\widehat\bx_t)\leq \mathcal{B}_1(D,\tilde M_0,\tilde M_1,m,\sigma_\epsilon,K,\delta)$ and $F_i(\bx_{t+1}) \leq \mathcal{B}_2(D,\tilde M_0,\tilde M_1,m,\sigma_\epsilon,K,\delta)$ for functions $\mathcal{B}_1(D,\tilde M_0,\tilde M_1,m,\sigma_\epsilon,K,\delta)=O(\frac{\log(K/\delta)}{\sqrt{K}})$ and $\mathcal{B}_1(D,\tilde M_0,\tilde M_1,m,\sigma_\epsilon,K,\delta)=O(\frac{\log(K/\delta)}{\sqrt{K}})$ with a probability of at least $1-\delta$. As a consequence, when $K$ is large enough (i.e. $K=\tilde O(\frac{1}{\hat\epsilon^4}\log(\frac{1}{\delta}))$) so that $\mathcal{B}_1\leq \hat\epsilon^2$ and $\mathcal{B}_2\leq \hat\epsilon^2$, Algorithm~\ref{alg:Yu} can be used as an oracle $\mathcal{A}$ for \eqref{eq:phit}. The complexity of Algorithm~\ref{alg:iqrc} using Algorithm~\ref{alg:Yu} as an oracle is therefore $\tilde O(\frac{1}{\epsilon^6})$. %Here, $\tilde O(\cdot)$ suppresses some logarithmic factors.
	\end{theorem}
	Since functions $\mathcal{B}_1$ and $\mathcal{B}_2$ are complicated, we put them in \eqref{eq:B1} and \eqref{eq:B2} in the supplementary file.
	
	%\section{Oracle for Smooth Deterministic Problem}
	\section{Numerical Experiments}
	\label{sec:exp}
%	In this section, we evaluate the numerical performance of the proposed methods on a multi-class Neyman-Pearson classification (mNPC) problem with nonconvex loss. Let $D_k \subset \mathbb{R}^d$ denote a group of data belonging to class $k$, for $k = 1, 2, \dots, K$. We train $K$ linear models $\bx_k$, $k=1,2,\dots K$, and then predict the class of data $\xi$ by $\argmax_{k=1,2,\dots,K} \bx_k^\top \xi$.  To achieve a high classification accuracy, we would like the value $\bx_k^\top \xi_k-\bx_l^\top \xi_k$ with $k\neq l$ to be positively large~\citep{weston1998multi,crammer2002learnability}, which can be achieved by minimizing the expected loss $\E\phi(\bx_k^\top \xi_k-\bx_l^\top \xi_k)$, where $\phi$ is a non-increasing  potentially non-convex loss function and $\E$ is the expectation taken over $\xi_k$. mNPC prioritizes minimizing the loss on one class, which is considered to be more important in practice, and then controls the losses on all other classes by solving  
In this section, we evaluate the numerical performance of the proposed methods on a multi-class Neyman-Pearson classification (mNPC) problem with nonconvex loss. Let $\xi_k$ for $k=1,2,\dots,K$ denote $K$ classes of data, each of which belongs to a subset of training data $D_k$. We train $K$ linear models $\bx_k$, $k=1,2,\dots K$, and then predict the class of data $\xi$ by $\argmax_{k=1,2,\dots,K} \bx_k^\top \xi$.  To achieve a high classification accuracy, the value $\bx_k^\top \xi_k-\bx_l^\top \xi_k$ needs to be positively large for any $k \neq l$ and $\xi_k \in D_k$~\citep{weston1998multi,crammer2002learnability}, which can be achieved by minimizing the following loss 
\small
\begin{align*}
	\frac{1}{|D_k|}\sum_{l \neq k}\sum_{\xi_k \in D_k} \phi(\bx_k^\top \xi_k-\bx_l^\top \xi_k),
\end{align*}
\normalsize
where $\phi$ is a non-increasing  potentially non-convex loss function. mNPC prioritizes minimizing the loss on one class, which is class 1 in our formulation, and then controls the loss on all other classes by solving  
    \small
	\begin{equation}
	\begin{aligned}
	\label{eq:NPclassification_multi}
	&\min_{\|\bx_k\|_2\leq\lambda,k=1,\dots,K}\frac{1}{|\mathcal{D}_1|}\sum_{l\neq k}\sum_{\xi\in\mathcal{D}_1}\phi(\bx_k^\top \xi-\bx_l^\top \xi), \\
	&\mathrm{s.t.}~\frac{1}{|\mathcal{D}_k|}\sum_{l\neq k}\sum_{\xi\in\mathcal{D}_k}\phi(\bx_k^\top \xi-\bx_l^\top \xi)\leq r_k,\quad k=2,3,\dots,K,
	\end{aligned}
    \end{equation}
    \normalsize
	where $r_k$ controls the loss of class $k$ and $\lambda$ is the regularization parameter. In the experiment, the function $\phi$ in~\eqref{eq:NPclassification_multi} is chosen as the sigmoid function $1/(1+\text{exp}(z))$. 
	
%	\begin{remark}
%	\label{eq:example}
%	It can be shown that Assumption~\ref{assume:stochastic}B holds for the problem above in some scenarios. To show this, it suffices to show that, in some scenarios,  there exists a solution $\bx$ such that $\|\bx\|_2\leq \lambda$ and $\frac{1}{|\mathcal{D}_k|}\sum_{l\neq k}\sum_{\xi\in\mathcal{D}_k}\phi(\bx_k^\top \xi-\bx_l^\top \xi) - r_k + \frac{\rho+\rho_\epsilon}{2}\|\bx-\bx_\epsilon\|^2 \leq -\sigma_\epsilon$ for any $\epsilon^2$-feasible point $\bx_\epsilon$.
%	\end{remark}

	We compare our IQRC method to the exact penalty method proposed in \citep{cartis2011evaluation}. Both methods are
implemented in Matlab on a 64-bit MacOS Catalina machine with a 2.90 Ghz Intel Core i7-6920HQ CPU and 16GB of memory. We conduct experiments on three LIBSVM multi-class classification datasets \emph{pendigits}, \emph{segment}, and \emph{usps}. \emph{pendigits} dataset has 7494 instances and 10 classes, while each instance is represented by a feature vector of dimension 16. \emph{segment} dataset has 2310 instances and 7 classes, and each instance is represented by a feature vector of dimension 7. \emph{usps} dataset has 7291 instances and 10 classes while each instance has 256 number of features. 
%The function $\phi$ in ~\eqref{eq:NPclassification_multi} is chosen as the sigmoid function $1/(1+\text{exp}(z))$. 
We choose $r_k = 4.5, k=2,\dots, K$ for \emph{pendigits} and \emph{usps}, and $r_k = 3, k=2,\dots, K$ for \emph{segment}. $\lambda$ is selected to be 0.1 for all datasets. For both algorithms in comparison, the initial solution $\bx = \mathbf{0}$ is chosen. It is easy to verify that $\bx = \mathbf{0}$ is a feasible solution given the $\lambda$ and $r_k$ we choose. For both algorithms, we tune hyper-parameters from a discrete set of choices.
	
%To implement exact penalty method proposed by \citep{cartis2011evaluation}, we need to reformulate problem~\ref{eq:NPclassification_multi} to the form of problem~\ref{eq:gco}, 
% To simplify notations, used in problem~\ref{eq:NPclassification_multi}, let $f_0(\bx)$ denote the objective function and $f_i(\bx) \leq 0$ for $i=1,\dots,2K-1$ denote all constraints including both functional constraints and compact set constraints. Then exact penalty method proposed by \citep{cartis2011evaluation} first derives a direction $\mathbf{s}_k$ by solving 
 To simplify notations, let $f_0(\bx)$ denote the objective function of problem~\eqref{eq:NPclassification_multi}. Similarly,  let $f_i(\bx) \leq 0$ for $i=1,\dots,2K-1$ denote all constraints of problem~\eqref{eq:NPclassification_multi} including both functional constraints and compact set constraints. Then exact penalty method proposed by \citep{cartis2011evaluation} first derives a direction $\mathbf{s}_k$ by solving 
    \small
	\begin{equation}
	\begin{aligned}
	   \mathbf{s}_{k}\in\argmin_{\|\mathbf{s}\|\leq \Delta_k}
    \big\{
    %\ell(\bx^{(k)},\mathbf{s}):=
    &f_0(\bx_{k})+\nabla f_0(\bx_{k})^\top\mathbf{s} \\ %+\rho\left\|\bA(\bx^{(k)}+\mathbf{s})-\bb\right\|
    + &p \sum_{i=1}^{2K-1}\left | \max \lbrace 0, f_i(\bx_{k})+\nabla f_i(\bx_{k})^\top\mathbf{s} \rbrace \right |
    \big\},
    \end{aligned}
    \label{eqn: trust_region_sk}
	\end{equation}
	\normalsize
	where $p > 0$ is the penalty parameter and $\Delta_k > 0$ is the radius. When $\mathbf{s}_k$ is derived, a trust region method is used to update the current solution $\bx_k$ by performing $\bx_{k+1} = \bx_{k} + \mathbf{s_k}$ if this update can significantly reduce the value of the function
	\small
    \begin{align*}
    	f_0(\bx) +&p \sum_{i=1}^{2K-1}\left |\max \lbrace 0, f_i(\bx) \rbrace \right |.
    \end{align*}
    \normalsize
	In our implementation, we reformulate problem \eqref{eqn: trust_region_sk} as a linear programming problem and then use the Matlab built-in solver to obtain $\mathbf{s}_k$. The exact penalty method in \citep{cartis2011evaluation} requires several hyper-parameters including a steering parameter $\xi$, an increase factor $\tau$ to update the penalty parameter, an initial penalty parameter $p_{-1}$, and the tolerance $\epsilon$. 
    %	Steering parameter are set to be 0.1 and initial penalty parameters $p_{-1}$ are selected to be $1/\xi$ for all datasets. We choose increase factor to be 10 and tolearnce $\epsilon = 0.001$. 
    After tunning $\xi$ and $\tau$, we set $\xi = 0.1, \tau = 100, p_{-1} = 1 / \xi$, and $\epsilon=0.001$ for \textit{pendigits} and \textit{segment}. For \textit{usps}, we choose $\xi = 0.02, \tau = 100, p_{-1} = 1 / \xi$, and $\epsilon=0.001$. The trust-region algorithm in \cite{cartis2011evaluation} specifies only the interval in which the trust-region radius should fall in, and there are many possible choices. In the experiment, we specifically follow the rule that $\Delta_{k+1} = \Delta_k$ if $r_k \geq \eta_2$, $\Delta_{k+1} = \gamma_2 \Delta_k$ if $r_k \in [\eta_1, \eta_2]$, and $\Delta_{k+1} = \gamma_1 \Delta_k$ if $r_k < \eta_1$.
    The trust region subproblem also requires several control parameters. For all datasets, we use $\Delta_0 = 1, \eta_1 = 0.3, \eta_2 = 0.7, \gamma_1 = 0.5, \gamma_2 = 1$. (See \cite{cartis2011evaluation} for the definitions of $\Delta_0, \eta_1, \eta_2, \gamma_1$, and $\gamma_2$).
	
	For our proposed method, the subproblem in the IQRC method is solved using the switching subgradient method (Algorithm~\ref{alg:MD}). After tunning regularization parameter $\hat\rho$ and inner iteration number $K$, we set $\hat{\rho} = 1$, $K = 20000$, and tolerance $\epsilon = 0.001$ for all datasets.
	
	\begin{figure}[t]
    \centering
    {\includegraphics[scale=.2]{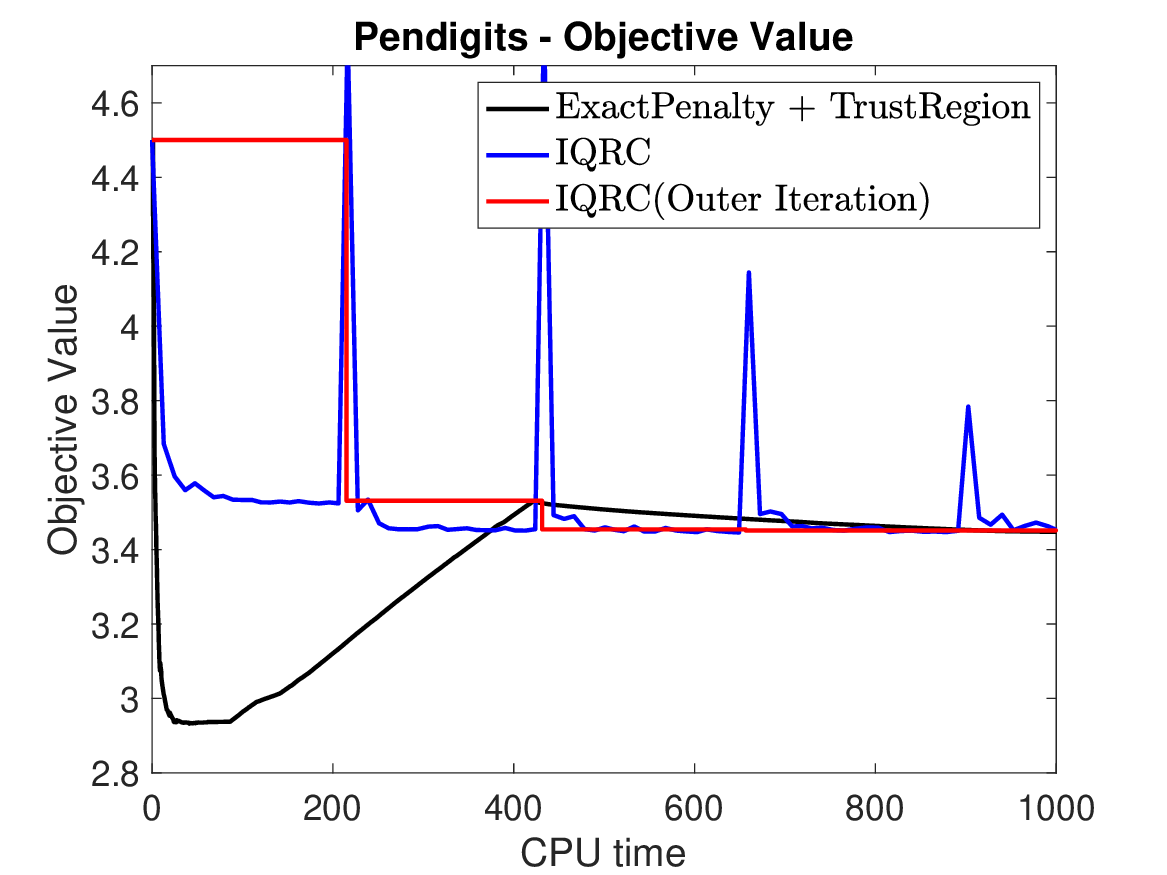}}
    {\includegraphics[scale=.2]{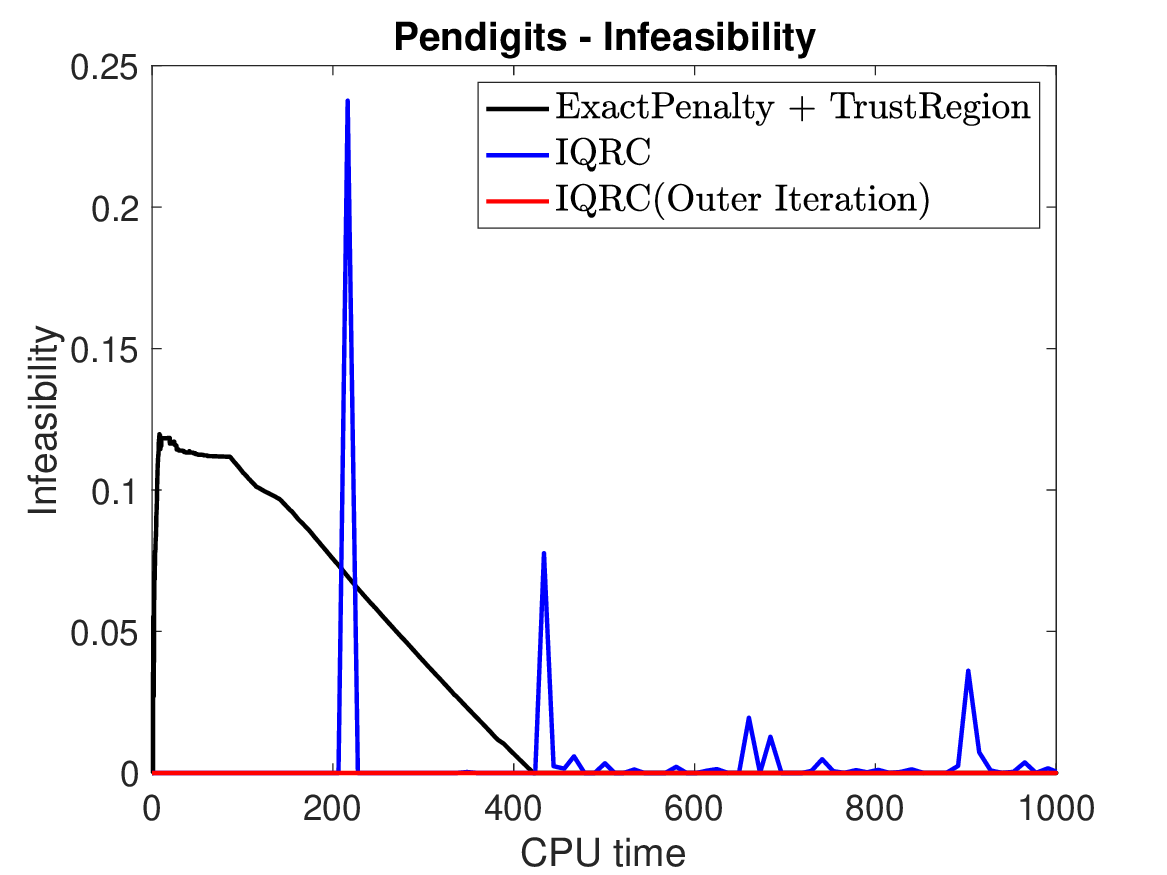}}
    {\includegraphics[scale=.2]{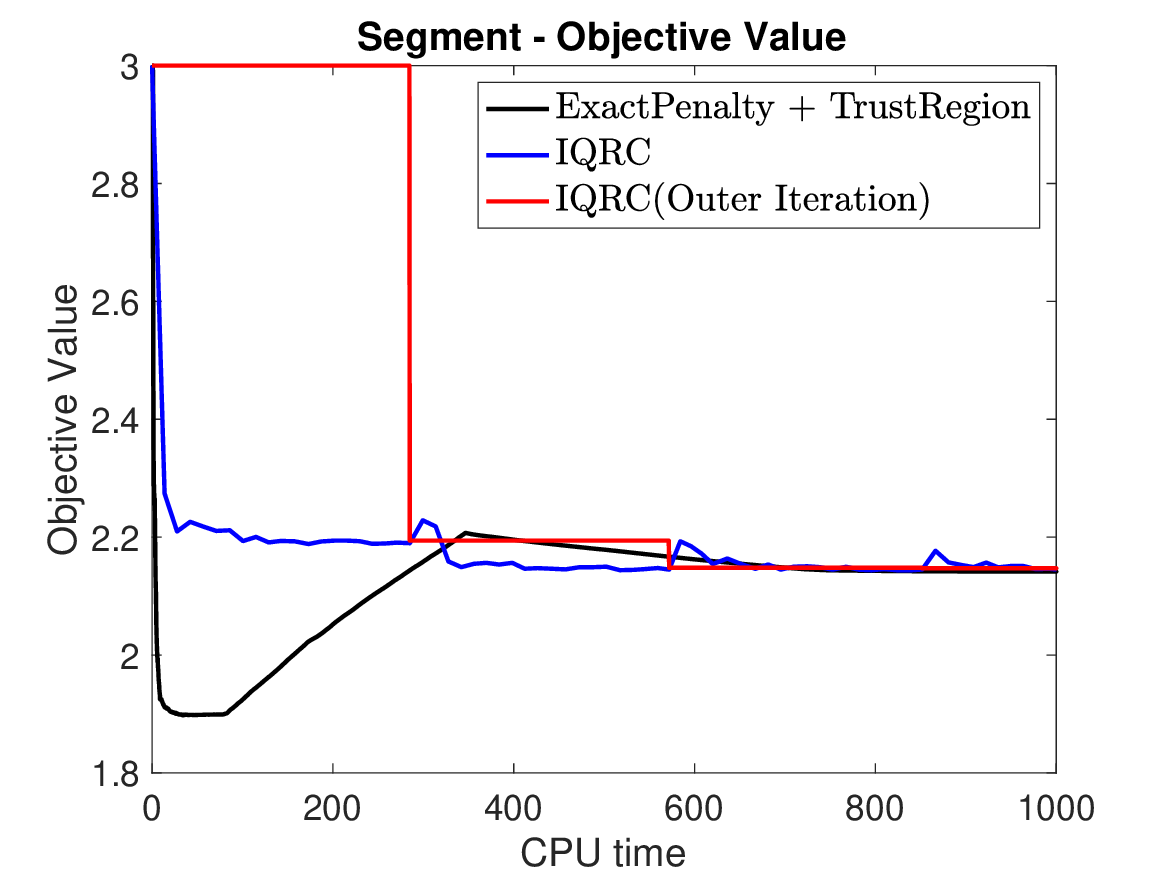}}
    {\includegraphics[scale=.2]{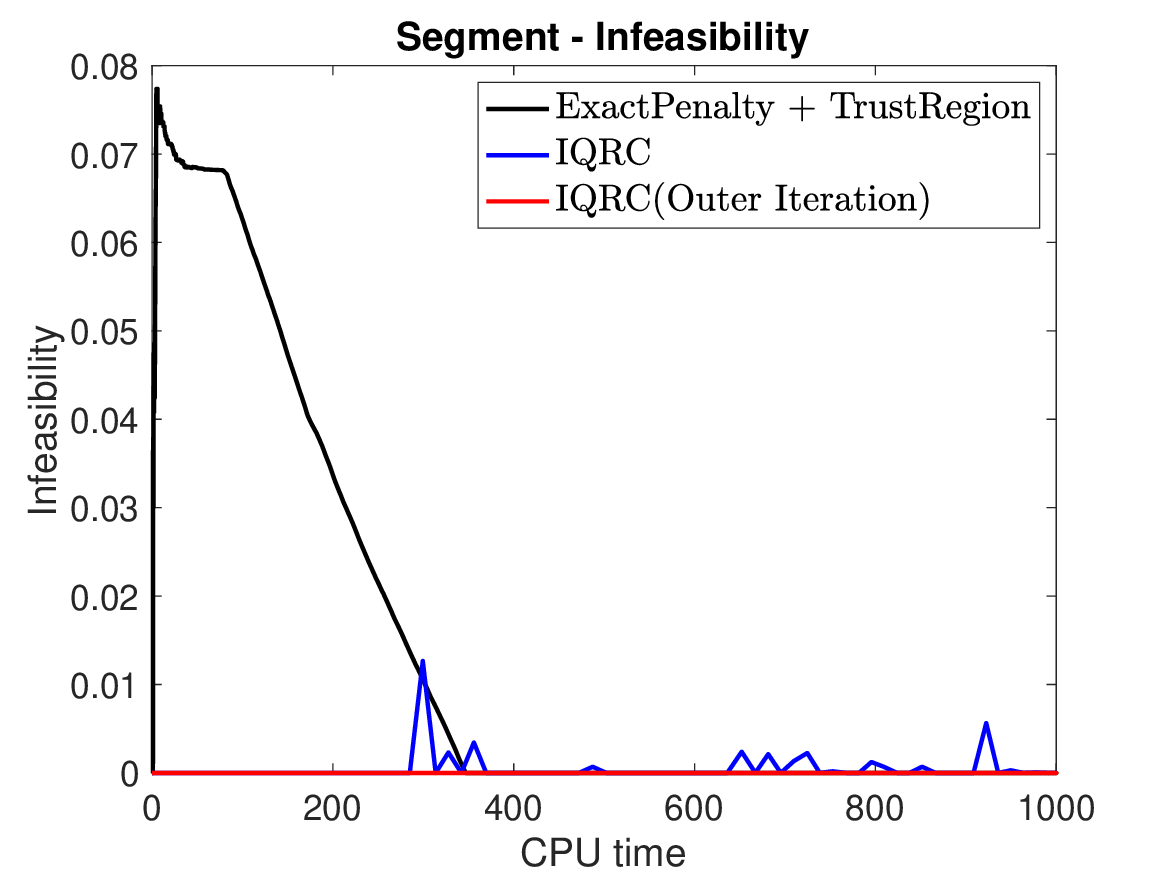}}
    {\includegraphics[scale=.2]{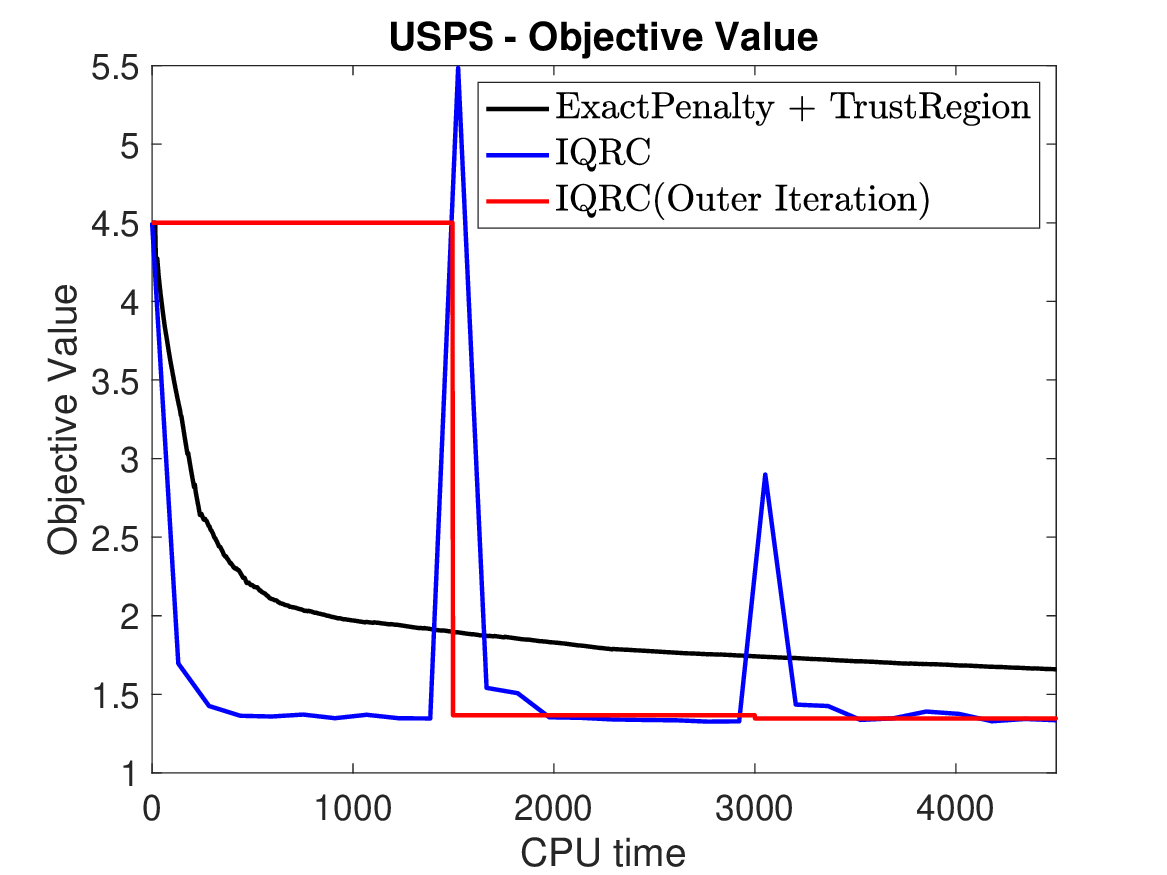}}
    {\includegraphics[scale=.2]{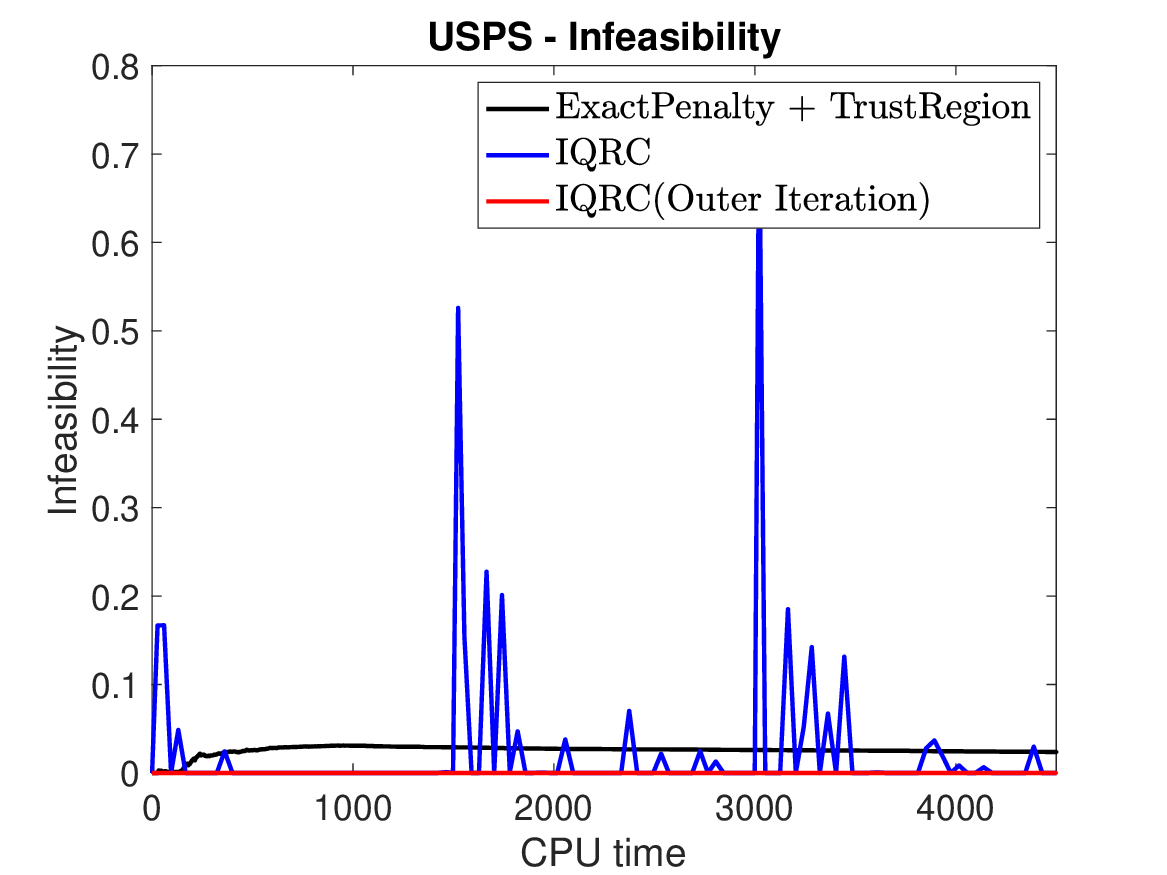}}
    \caption{Comparison between the IQRC method and the exact penalty method by \cite{cartis2011evaluation} for solving multi-class Neyman-Pearson classification problem \eqref{eq:NPclassification_multi}.} 
	\label{fig:neyman_pearson}
    \end{figure}
    
    The numerical results are presented in Figure \ref{fig:neyman_pearson}. The $x$-axis represents the CPU time that each algorithm took. The $y$-axis on left column of Figure  \ref{fig:neyman_pearson} represents the objective value of \eqref{eq:NPclassification_multi} and the $y$-axis on right column of Figure  \ref{fig:neyman_pearson} represents infeasibility, i.e., $\max \lbrace \max_{i=1,\dots,m}f_i(\bx), \max_{k=1,\dots,K} \|\bx_k\|-\lambda, 0 \rbrace  $, of the iterates. The red line shows the performance of the outer iteration solution $\bx_t$ in Algorithm \ref{alg:iqrc}, which is the solution users need in practice. The blue line evaluates the performance of inner iteration $\bz_t$ in Algorithm \ref{alg:MD}. We show it here only for reader's curiosity. The black line represents the performance of the exact penalty method proposed by \cite{cartis2011evaluation}.  We conclude from Figure~\ref{fig:neyman_pearson} that, for these three instances, the IQRC method outperformed the exact penalty method in terms of the capability of reducing the objective value and infeasiblity.

	\section{Conclusion}
	\label{sec:con}
	Continuous optimization models with nonlinear constraints have been widely used in many areas including machine learning, statistics and operations research. When non-convex functional constraints appear, even finding a feasible solution is challenging. In this paper, we proposed a class of quadratically regularized subgradient method which can find a nearly stationary point for functional constrained non-convex problems.  The complexity to find a nearly stationary point for both deterministic case and stochastic case are derived: (i) when each function $f_i$ is a deterministic function, we proposed a new variant of switching subgradient to solve strongly convex subproblem and the total complexity of Algorithm~\ref{alg:iqrc} for finding a nearly $\epsilon$-stationary point is $\mathcal{O}(\frac{1}{\epsilon^4})$, (ii) when each function $f_i$ is given as an expectation of a stochastic function, we analyzed stochastic subgradient method by \cite{yu2017online} for solving subproblem and the total complexity of Algorithm~\ref{alg:iqrc} for finding a nearly $\epsilon$-stationary point is $\tilde{\mathcal{O}}(\frac{1}{\epsilon^6})$.

	\subsubsection*{Acknowledgments}
	%
	%Use unnumbered third level headings for the acknowledgments. All acknowledgments
	%go at the end of the paper. Do not include acknowledgments in the anonymized
	%submission, only in the final paper.
	T. Yang is partially supported by National Science Foundation CAREER Award 1844403.
	
    \bibliographystyle{icml2020}
	\bibliography{SIP,references,iclr2019_conference}

\newpage
\appendix
\onecolumn
\section{Appendix}
\label{sec:appendix}
In this section, we provide the proofs for the theoretical results in the paper.
\subsection{Proof of Lemma~\ref{thm:boundlambda}}
\begin{proof}
	By KKT conditions, it holds that $\lambda_t \geq 0$ and $\lambda_t \left( g(\widehat{\bx}_t) + \frac{\hat\rho}{2}\|\widehat{\bx}_t -\bx_t\|^2 \right) = 0$. If $\lambda_t = 0$, there is nothing to show. So, we focus on the case that $\lambda_t > 0$ and $g(\widehat{\bx}_t) + \frac{\hat\rho}{2}\|\widehat{\bx}_t -\bx_t\|^2 = 0$.  Note that $\bx_0$ is an $\epsilon^2$-feasible solution. Using the definitions of $\mathcal{A}(\bx_t,\hat\rho,\hat\epsilon,\delta/T)$ and $\hat\epsilon$ and the union bound, we can show that the iterate $\bx_t$ generated by Algorithm~\ref{alg:iqrc} is an $\epsilon^2$-feasible solution for any $t$ with a probability of at least $1-\delta$.
	
	Let $\tilde\bx_t \equiv \argmin\limits_{\bx \in \X} \lbrace g(\bx) + \frac{\hat{\rho}}{2} \| \bx - \bx_t \|^2 \rbrace $. According to Assumption~\ref{assume:stochastic}B, the fact that $\bx_t$ is $\epsilon^2$-feasible, and the fact that $\hat\rho\leq \rho+\rho_\epsilon$, we have 
	\begin{eqnarray}
		\label{eq:slator1}
		-\sigma_\epsilon \geq \min\limits_{\bx \in \X}g(\bx) + \frac{\rho+\rho_\epsilon}{2}\| \bx - \bx_t \|^2
		\geq\min\limits_{\bx \in \X}g(\bx) + \frac{\hat{\rho}}{2}\| \bx - \bx_t \|^2
		= g(\tilde\bx_t) + \frac{\hat{\rho}}{2}\| \tilde\bx_t  - \bx_t \|^2.
	\end{eqnarray}
	As a result, the Lagrangian multiplier $\lambda_t$  is well-defined and satisfies the optimality condition below together with $\widehat\bx_t$:
	\begin{align}
		\label{eq:optcond1}
		\mathbf{0}\in\partial f(\widehat\bx_t) + \hat{\rho} (\widehat\bx_t -\bx_t ) + \lambda_t (\partial g(\widehat\bx_t) + \hat{\rho} (\widehat\bx_t -\bx_t )) +\widehat\bzt_t,
	\end{align}
	for some $\widehat\bzt_t\in \mathcal{N}_\X(\widehat\bx_t)$.

	Since $g(\bx) + \frac{\hat{\rho}}{2}\| \bx - \bx_t \|^2+\mathbf{1}_{\X}(\bx) $ is $(\hat{\rho} - \rho)$-strongly convex in $\bx$ and $\frac{\widehat\bzt_t}{\lambda_t}\in \mathcal{N}_\X(\widehat\bx_t)=\partial\mathbf{1}_{\X}(\widehat\bx_t)$,  we have
	\small
	\begin{align*}
		g(\tilde\bx_t) + \frac{\hat{\rho}}{2}\| \tilde\bx_t  - \bx_t \|^2 &\geq g(\widehat\bx_t) + \frac{\hat{\rho}}{2}\| \widehat\bx_t -\bx_t \|^2 + \langle \partial g(\widehat\bx_t) + \hat{\rho} (\widehat\bx_t -\bx_t )+\frac{\widehat\bzt_t}{\lambda_t}, \tilde\bx_t - \widehat\bx_t \rangle + \frac{\hat{\rho} - \rho}{2} \| \tilde\bx_t - \widehat\bx_t \|^2 \\
		&= \langle \partial g(\widehat\bx_t) + \hat{\rho} (\widehat\bx_t -\bx_t )+\widehat\bzt_t/\lambda_t, \tilde\bx_t - \widehat\bx_t \rangle + \frac{\hat{\rho} - \rho}{2} \| \tilde\bx_t - \widehat\bx_t \|^2.
	\end{align*}
	\normalsize
	Applying \eqref{eq:slator1} to the inequality above and arranging terms give
	\begin{align*}
		-\sigma_{\epsilon} - \frac{(\hat{\rho} - \rho)\| \tilde\bx_t - \widehat\bx_t \|^2}{2}  &\geq \langle \partial g(\widehat\bx_t) + \hat{\rho} (\widehat\bx_t -\bx_t )+\widehat\bzt_t/\lambda_t, \tilde\bx_t - \widehat\bx_t \rangle \\
		&\geq - \frac{ \| \partial g(\widehat\bx_t) + \hat{\rho} (\widehat\bx_t -\bx_t )+\widehat\bzt_t/\lambda_t \|^2 }{2(\hat{\rho} - \rho)} - \frac{(\hat{\rho} - \rho)\|\tilde\bx_t - \widehat\bx_t \|^2}{2},
	\end{align*}
	which implies $\| \partial g(\widehat\bx_t) + \hat{\rho} (\widehat\bx_t -\bx_t )+\widehat\bzt_t/\lambda_t \|^2 \geq 2 \sigma_{\epsilon} (\hat{\rho} - \rho)$. 
	
	Using this lower bound on $\| \partial g(\widehat\bx_t) + \hat{\rho} (\widehat\bx_t -\bx_t )+\widehat\bzt_t/\lambda_t \|^2$ and \eqref{eq:optcond1}, we have that
	\small
	\begin{align*}
		\lambda_t = \frac{\| \partial f(\widehat\bx_t) + \hat{\rho} (\widehat\bx_t -\bx_t ) \| }{\| \partial g(\widehat\bx_t) + \hat{\rho} (\widehat\bx_t -\bx_t ) +\widehat\bzt_t/\lambda_t\| } 
		%\leq \frac{\| \partial f(\widehat\bx_t) -\partial g(\widehat\bx_t)\| +\|\partial g(\widehat\bx_t) + \hat{\rho} (\widehat\bx_t -\bx_t ) \| }{\| \partial g(\widehat\bx_t) + \hat{\rho} (\widehat\bx_t -\bx_t ) \| } 
		\leq  \frac{M+\hat\rho D }{\sqrt{2 \sigma_{\epsilon} (\hat{\rho} - \rho)}}
	\end{align*}
	\normalsize
	for all $t$ with a probability of at least $1-\delta$, where we have used Assumption~\ref{assume:stochastic}C and Assumption~\ref{assume:stochastic}F in the inequality.
\end{proof}

\subsection{Proof of Theorem~\ref{thm:main}}

\begin{proof}
	Since $\bx_{t+1}=\mathcal{A}(\bx_t,\hat\rho,\hat\epsilon,\delta/T)$, the definition of $\mathcal{A}$ and the union bound imply that the following inequalities hold for $t=0,\dots,T-1$ with a probability of at least $1-\delta$. 
	\small
	\begin{align}
		f(\bx_{t+1}) + \frac{\hat{\rho}}{2}\| \bx_{t+1} - \bx_t \|^2 - f(\widehat\bx_t) - \frac{\hat{\rho}}{2}\| \widehat\bx_t - \bx_t \|^2 \leq \hat\epsilon^2,\quad
		g(\bx_{t+1}) + \frac{\hat{\rho}}{2}\|\bx_{t+1} - \bx_t \|^2 \leq \hat\epsilon^2. \label{PGSMeq1}
	\end{align}
	\normalsize
	Let $\lambda_t$ be the optimal Lagrangian multiplier corresponding to $\widehat\bx_t$. Then $\widehat\bx_t$ is also the optimal solution of the Lagrangian function 
	$\mathcal{L}(\bx)\equiv f(\bx) + \frac{\hat{\rho}}{2}\|\bx - \bx_t \|^2 + \lambda_t(g(\bx) + \frac{\hat{\rho}}{2}\| \bx- \bx_t\|^2)$. Since $\mathcal{L}(\bx)$ is $(1+\lambda_t)(\hat\rho-\rho)$-strongly convex, we have
	\small
	\begin{eqnarray}
		\nonumber
		\frac{(1+\lambda_t)(\hat{\rho}-\rho)}{2}\|\bx_t - \widehat\bx_t\|^2 &\leq& f(\bx_t) + \frac{\hat{\rho}}{2}\|\bx_t - \bx_t \|^2 + \lambda_t(g(\bx_t) + \frac{\hat{\rho}}{2}\| \bx_t - \bx_t\|^2) \\
		&&- \left[f(\widehat\bx_t) + \frac{\hat{\rho}}{2}\|\widehat\bx_t - \bx_t \|^2 + \lambda_t(g(\widehat\bx_t) + \frac{\hat{\rho}}{2}\| \widehat\bx_t - \bx_t\|^2)\right] \nonumber \\
		&=& f(\bx_t) - f(\widehat\bx_t) + \lambda_t g(\bx_t) - \frac{\hat{\rho}}{2}\| \widehat\bx_t - \bx_t\|^2, \label{PGSMeq2}
	\end{eqnarray}
	\normalsize
	where we use the complementary slackness, i.e., $\lambda_t(g(\widehat\bx_t) + \frac{\hat{\rho}}{2}\| \widehat\bx_t - \bx_t\|^2)=0$ in the equality above.
	Organizing the terms in the first inequality of (\ref{PGSMeq1}), we get
	\begin{align}
		f(\bx_{t+1}) &\leq f(\widehat\bx_t) + \hat\epsilon^2 + \frac{\hat{\rho}}{2}\| \widehat\bx_t - \bx_t\|^2 - \frac{\hat{\rho}}{2}\| \bx_{t+1} - \bx_t \|^2 \nonumber \\
		&\leq f(\widehat\bx_t) + \hat\epsilon^2 + f(\bx_t) - f(\widehat\bx_t) + \lambda_t g(\bx_t) - \frac{(1+\lambda_t)(\hat{\rho}-\rho)}{2}\|\bx_t - \widehat\bx_t \|^2 \nonumber \\
		&= f(\bx_t) + \lambda_t g(\bx_t) - \frac{(1+\lambda_t)(\hat{\rho}-\rho)}{2}\|\bx_t - \widehat\bx_t \|^2 + \hat\epsilon^2 \nonumber 
	\end{align}
	where second inequality is because of (\ref{PGSMeq2}). The inequality  above can be written as
	\begin{align}
		\frac{(1+\lambda_t)(\hat{\rho}-\rho)}{2}\|\bx_t - \widehat\bx_t \|^2 \leq f(\bx_t) - f(\bx_{t+1}) + \lambda_t g(\bx_t) + \hat\epsilon^2 \label{PGSMeq3}
	\end{align}
	% Divided both sides of (\ref{PGSMeq3}) by $\frac{(1+\lambda_t)(\hat{\rho}-\rho)}{2}$, we get 
	% \begin{align*}
		%     \|\bx_t - \widehat\bx_t \|^2 &\leq \frac{2}{(1+\lambda_t)(\hat{\rho}-\rho)}(f(\bx_t) - f(\bx_{t+1})) + \frac{2\lambda_t}{(1+\lambda_t)(\hat{\rho}-\rho)}g(\bx_t) \\
		%     &\leq \frac{2}{\hat{\rho} - \rho}(f(\bx_t)-f(\bx_{t+1})) + \frac{2}{\hat{\rho}-\rho}g(\bx_t) + \frac{2}{\hat{\rho}-\rho}\hat\epsilon
		% \end{align*}
	% Note that $g(\bx_t) \leq g(\bx_t)+\frac{\hat{\rho}}{2}\|\bx_t - \bx_{t-1}\|^2 \leq \hat\epsilon $ because of setting of the algorithm. So we have
	% \begin{align}
		%     \| \bx_t - \widehat\bx_t \|^2 &\leq \frac{2}{\hat{\rho}-\rho}(f(\bx_t) - f(\bx_{t+1}) + \frac{4}{\hat{\rho}-\rho}\hat\epsilon \nonumber \\
		%     &\leq \frac{2}{\hat{\rho}-\rho}(f(\bx_t) - f(\bx_{t+1}) + \frac{4(\hat{\rho}-\rho)\hat\epsilon}{8(\hat{\rho}-\rho)(\hat{\rho}+\lambda_t \hat{\rho})^2} \nonumber \\
		%     &= \frac{2}{\hat{\rho}-\rho}(f(\bx_t) - f(\bx_{t+1}) + \frac{\hat\epsilon}{2(\hat{\rho}+\lambda_t \hat{\rho})^2} \label{PGSMeq4}
		% \end{align}
	% Sum up inequality (\ref{PGSMeq4}) from $t = 1,\dots,T$ and then divide both sides by $T$,
	% \begin{align*}
		%     \frac{1}{T}\sum_{t=1}^T \|\bx_t - \widehat\bx_t \|^2 &\leq  \frac{2}{T(\hat{\rho} - \rho)}(f(\bx_0) - f_{\text{lsb}}) + \frac{\hat\epsilon}{2(\hat{\rho}+\lambda (\hat{\rho}))^2} \\
		%     &\leq \frac{\hat\epsilon}{(\hat{\rho}+\lambda \hat{\rho})^2}
		% \end{align*}
	% where second inequality follows definition of $T$.
	Summing up inequality (\ref{PGSMeq3}) from $t=0,1,\dots,T-1$, we have
	\begin{align*}
		\sum_{t = 0}^{T-1} \frac{(1+\lambda_t)(\hat{\rho}-\rho)}{2}\| \bx_t - \widehat\bx_t \|^2 \leq f(\bx_0) - f_{\text{lb}} + \sum_{t=0}^{T-1} \lambda_t g(\bx_t) + T \hat\epsilon^2,
	\end{align*}
	where $f_{\text{lb}}$ is introduced in Assumption~\ref{assume:stochastic}D.
	Note that $g(\bx_t) \leq g(\bx_t)+\frac{\hat{\rho}}{2}\|\bx_t - \bx_{t-1}\|^2 \leq \hat\epsilon^2 $ because of the property of $\mathcal{A}$. So we have
	\begin{align}
		\sum_{t = 0}^{T-1}\frac{(\hat{\rho}-\rho)}{2}\| \bx_t - \widehat\bx_t \|^2 \leq\sum_{t = 0}^{T-1}\frac{(1+\lambda_t)(\hat{\rho}-\rho)}{2}\| \bx_t - \widehat\bx_t \|^2 \leq f(\bx_0) - f_{\text{lb}} + \sum_{t=0}^{T-1} \lambda_t \hat\epsilon^2 +T\hat\epsilon^2. \nonumber
		%= f(\bx_0) - f_{\text{lsb}} + \sum_{t=0}^{T-1} (1+\lambda_t) \hat\epsilon. \nonumber
	\end{align}
	Dividing both sides by $T(\hat{\rho}-\rho)/2$, we have
	\small
	\begin{align*}
		\mathbb{E}_R \| \bx_R - \widehat\bx_R \|^2=\frac{1}{T}\sum_{t = 0}^{T-1} \| \bx_t - \widehat\bx_t \|^2 
		&\leq \frac{2(f(\bx_0) - f_{\text{lb}})}{T(\hat{\rho}-\rho)} + \frac{2}{T(\hat{\rho}-\rho)}\sum_{t=0}^{T-1} (1+\lambda_t) \hat\epsilon^2 \nonumber \\
		&\leq \frac{2(f(\bx_0) - f_{\text{lb}})}{T(\hat{\rho}-\rho)}+\frac{2\hat\epsilon^2}{(\hat{\rho}-\rho)}\left(\frac{M+\hat\rho D }{\sqrt{2 \sigma_{\epsilon} (\hat{\rho} - \rho)}}+1\right)\\
		&\leq \frac{\epsilon^2}{2}+\frac{\epsilon^2}{2}=\epsilon^2
	\end{align*}
	\normalsize
	with a probability of at least $1-\delta$, 	where the second inequality is by Lemma~\ref{thm:boundlambda} and the last inequality follows the definitions of $T$ and $\hat\epsilon$.
\end{proof}

\subsection{Proof of Theorem~\ref{thm:mdconverge}}
\begin{proof}
	For simplicity of notation, we defined $\mu:=\hat\rho-\rho$.
	Let $J:=\{0,1,\dots,K-1\}\backslash I$ where $I$ is generated in  Algorithm~\ref{alg:MD} when it terminates.
	
	Suppose $k\in I$, namely, $G(\bz_k)\leq \hat\epsilon^2$ is satisfied in iteration $k$.
	Algorithm~\ref{alg:MD} will update $\bz_{k+1}$ using $F'(\bz_k)$.
	Following the standard analysis of subgradient decent method, we can get  
	\begin{align}
		F(\bz_k) - F(\widehat\bx_t) &\leq \gamma_k (M^2+\hat{\rho}^2D^2) + (\frac{1}{2\gamma_k} - \frac{\mu}{2})\|\bz_k - \widehat\bx_t\|^2 - \frac{\|\bz_{k+1} - \widehat\bx_t \|^2}{2\gamma_k} \nonumber \\
		&= \frac{2(M^2+\hat{\rho}^2D^2)}{\mu(k+2)} + (\frac{\mu(k+2)}{4}-\frac{2\mu}{4}) \| \bz_k - \widehat\bx_t\|^2 - \frac{\mu(k+2)}{4}\| \bz_{k+1} - \widehat\bx_t\|^2 \nonumber\\
		&= \frac{2(M^2+\hat{\rho}^2D^2)}{\mu(k+2)} + \frac{\mu k}{4}\|\bz_k - \widehat\bx_t \|^2 - \frac{\mu(k+2)}{4}\| \bz_{k+1} - \widehat\bx_t\|^2 \label{MDeq1}
	\end{align}
	Multiplying $k+1$ to the both sides of \ref{MDeq1}, we can get 
	\small
	\begin{align}
		(k+1)(F(\bz_k) - F(\widehat\bx_t)) &\leq \frac{2(M^2+\hat{\rho}^2D^2)(k+1)}{\mu(k+2)} + \frac{\mu k (k+1)}{4}\|\bz_k - \widehat\bx_t\|^2 - \frac{\mu (k+1)(k+2)}{4}\|\bz_{k+1} - \widehat\bx_t\|^2 \nonumber \\
		&\leq \frac{2(M^2+\hat{\rho}^2D^2)}{\mu} + \frac{\mu k (k+1)}{4}\|\bz_k - \widehat\bx_t\|^2 - \frac{\mu (k+1)(k+2)}{4}\|\bz_{k+1} - \widehat\bx_t\|^2 \label{MDFside}
	\end{align}
	\normalsize
	
	Suppose $k\in J$, namely, $G(\bz_k)\leq \hat\epsilon^2$ is not satisfied in iteration $k$. Algorithm~\ref{alg:MD} will update $\bz_{k+1}$ using $G'(\bz_k)$.
	Similarly, we can get 
	\small
	\begin{align}
		(k+1)(G(\bz_k) - G(\widehat\bx_t)) \leq \frac{2(M^2+\hat{\rho}^2D^2)}{\mu} + \frac{\mu k (k+1)}{4}\|\bz_k - \widehat\bx_t\|^2 - \frac{\mu (k+1)(k+2)}{4}\|\bz_{k+1} - \widehat\bx_t\|^2 \label{MDGside}
	\end{align}
	\normalsize
	
	Summing up inequalities (\ref{MDFside}) and (\ref{MDGside}) from $k=0,\dots,K-1$ and dropping the non-negative terms, we obtain
	\small
	\begin{align}
		\sum\limits_{k \in I} (k+1)(F(\bz_k) - F(\widehat\bx_t)) + \sum\limits_{k \in J} (k+1)(G(\bz_k) - G(\widehat\bx_t)) &\leq \frac{2K(M^2+\hat{\rho}^2D^2)}{\mu}
	\end{align}
	\normalsize
	Because $G(\bz_k)>\hat\epsilon^2$ when $k\in J $ and   $G(\widehat\bx_t)\leq 0$, the inequality above implies
	\small
	\begin{align}
		\sum\limits_{k \in I} (k+1)(F(\bz_k) - F(\widehat\bx_t)) + \sum\limits_{k \in J} (k+1)\hat\epsilon^2 &\leq \frac{2K(M^2+\hat{\rho}^2D^2)}{\mu}
	\end{align}
	\normalsize
	Rearranging terms gives
	\small
	\begin{align}
		\sum_{k \in I}(k+1)(F(\bz_k) - F(\widehat\bx_t)) &\leq \sum_{k \in I}(k+1)\hat\epsilon^2 - \sum_{k=0}^{K-1} (k+1)\hat\epsilon^2 + \frac{2K(M^2+\hat{\rho}^2D^2)}{\mu}\nonumber \\
		&\leq \sum_{k \in I}(k+1)\hat\epsilon^2 - \frac{K(K+1)}{2}\hat\epsilon^2 + \frac{2K(M^2+\hat{\rho}^2D^2)}{\mu}. \nonumber
	\end{align}
	\normalsize
	Given that $K \geq \frac{4(M^2+\hat{\rho}^2D^2)}{\mu \hat\epsilon^2}$, the summation of the last two terms in the inequality above is non-positive. As a result, we have
	\begin{align}
		\sum_{k \in I}(k+1)(F(\bz_k) - F(\widehat\bx_t)) \leq \sum_{k \in I}(k+1)\hat\epsilon^2 \nonumber
	\end{align}
	Dividing both sides by $\sum_{k \in I} (k+1)$ and using the convexity of $F$, we obtain $F(\bx_{t+1}) - F(\widehat\bx_t) \leq \hat\epsilon^2$. As the same time, the convexity of $G$ ensures $G(\bx_{t+1}) \leq \frac{\sum_{k \in I}(k+1)G(\bz_k)}{\sum_{k \in I}(k+1)} \leq \hat\epsilon^2$.
	
	Hence, Algorithm~\ref{alg:MD} can be used as an oracle to solve \eqref{eq:phit} and the complexity of Algorithm~\ref{alg:iqrc} will be 
	$$
	TK=O\left(\frac{(f(\bx_0) - f_{\text{lb}})(M^2+\hat{\rho}^2D^2)}{\epsilon^4(\hat\rho-\rho)^3} \Big(\frac{M+\hat\rho D }{\sqrt{ \sigma_{\epsilon} (\hat{\rho} - \rho)}}+1\Big)\right).
	$$
	Note that,  Algorithm~\ref{alg:MD} is deterministic so that the complexity above does not depend on $\delta$.
\end{proof}

\subsection{Proof of Theorem~\ref{thm:Yu}}
\begin{proof}
	According to Assumption~\ref{assume:stochastic}B and the factor that $\bx_t$ is $\epsilon^2$-feasible with a high probability, Assumption 2 (The Slater's condition) in \cite{yu2017online} holds for the subproblem~\eqref{eq:phit} with a high probability. According to Theorem 4 in \cite{yu2017online}, Algorithm~\ref{alg:Yu} guarantees 
	\begin{eqnarray}
		\label{eq:YuObjective}
		F(\bx_{t+1})-F(\widehat\bx_t)\leq \mathcal{B}_1(D,\tilde M_0,\tilde M_1,m,\sigma_\epsilon,K,\delta)
	\end{eqnarray}
	with a probability of at least $1-\delta$, where
	\small
	\begin{eqnarray}
		\label{eq:B1}
		&&\mathcal{B}_1(D,\tilde M_0,\tilde M_1,m,\sigma_\epsilon,K,\delta)\\\nonumber
		&\equiv&\frac{D^2+\tilde M_1^2/4+(\tilde M_0+\sqrt{m}\tilde M_1D)^2/2+\log^{0.5}\big(\frac{1}{\delta}\big)\tilde M_0\Lambda(D,\tilde M_0,\tilde M_1,m,\sigma_\epsilon,K,\delta)}{\sqrt{K}},
	\end{eqnarray}
	\small
	\small
	\begin{eqnarray}
		\label{eq:Lambda1}
		\Lambda(D,\tilde M_0,\tilde M_1,m,\sigma_\epsilon,K,\delta)&\equiv&\frac{\sigma_\epsilon }{2}+(\tilde M_0+\sqrt{m}\tilde M_1D)
		+\frac{2D^2 }{\sigma_\epsilon}
		+\frac{2\tilde M_1 D+(\tilde M_0+\sqrt{m}\tilde M_1D)^2}{\sigma_\epsilon}\\\nonumber
		&&+\tilde\Lambda(D,\tilde M_0,\tilde M_1,m,\sigma_\epsilon,K,\delta)
		+\frac{8(\tilde M_0+\sqrt{m}\tilde M_1D)^2}{\sigma_\epsilon}\log\big(\frac{2K}{\delta}\big)
		=O(\log(K/\delta)),
	\end{eqnarray}
	\normalsize
	and
	\small
	\begin{eqnarray*}
		\tilde\Lambda(D,\tilde M_0,\tilde M_1,m,\sigma_\epsilon,K,\delta)&\equiv&\frac{8(\tilde M_0+\sqrt{m}\tilde M_1D)^2}{\sigma_\epsilon}\log\left[1+\frac{32(\tilde M_0+\sqrt{m}\tilde M_1D)^2}{\sigma_\epsilon^2}
		\exp\left(\frac{\sigma_\epsilon}{8(\tilde M_0+\sqrt{m}\tilde M_1D)}\right)\right].
	\end{eqnarray*}
	\normalsize
	
	According to equation (22) in \cite{yu2017online}, Algorithm~\ref{alg:Yu} guarantees 
	\small
	\begin{eqnarray}
		\label{eq:YuConstraint}
		F_i(\bx_{t+1})&\leq& \frac{\|(Q_K^1,Q_K^2,\dots,Q_K^m)\|}{K}+\frac{\tilde M_1^2}{\sqrt{K}}+\frac{\sqrt{m}\tilde M_1^2}{2K^2}\sum_{k=0}^{K-1}\|(Q_k^1,Q_k^2,\dots,Q_k^m)\|
	\end{eqnarray}
	\normalsize
	for $i=1,\dots,m$.
	It is also shown in  Theorem 3 in \cite{yu2017online} that 
	\small
	\begin{eqnarray}
		\label{eq:YuQbound}
		\|(Q_k^1,Q_k^2,\dots,Q_k^m)\|&\leq& \sqrt{K}\Lambda(D,\tilde M_0,\tilde M_1,m,\sigma_\epsilon,K,\delta)
	\end{eqnarray}
	\normalsize
	for $k=0,1,\dots,K$ with a probability of at least $1-\delta$.
	%	\small
	%	\begin{eqnarray}
		%	\label{eq:YuQbound}
		%	\mathbb{E}\|(Q_k^1,Q_k^2,\dots,Q_k^m)\|&\leq& \frac{\sigma_\epsilon\sqrt{K}}{2}+2(\tilde M_0+\sqrt{m}\tilde M_1 D)\sqrt{K}+\frac{2\sqrt{K}D^2}{\sigma_\epsilon}+\frac{2\sqrt{K}\tilde M D+(\tilde M_0+\sqrt{m}\tilde M_1D)^2}{\sigma_\epsilon}\nonumber\\
		%	&&+\frac{8\sqrt{K}(\tilde M_0 + \sqrt{m} \tilde M_1D)^2}{\sigma_\epsilon}\log\left(\frac{32(\tilde M_0 + \sqrt{m} \tilde M_1D)^2}{\sigma_\epsilon^2}\right)\nonumber\\
		%	&=&\sqrt{K}C_1.
		%	\end{eqnarray}
	%	\normalsize
	Applying \eqref{eq:YuQbound} to \eqref{eq:YuConstraint} and organizing terms, we obtain
	\small
	\begin{eqnarray}
		\label{eq:YuConstraintnew}
		F_i(\bx_{t+1})&\leq& \mathcal{B}_2(D,\tilde M_0,\tilde M_1,m,\sigma_\epsilon,K,\delta)
	\end{eqnarray}
	\normalsize
	with a probability of at least $1-\delta$, where 
	\small
	\begin{eqnarray}
		\label{eq:B2}
		&&\mathcal{B}_2(D,\tilde M_0,\tilde M_1,m,\sigma_\epsilon,K,\delta)\\\nonumber
		&\equiv&\frac{\Lambda(D,\tilde M_0,\tilde M_1,m,\sigma_\epsilon,K,\delta)+\tilde M_1^2+\Lambda(D,\tilde M_0,\tilde M_1,m,\sigma_\epsilon,K,\delta)\sqrt{m}\tilde M_1^2/2}{\sqrt{K}}
	\end{eqnarray}
	\normalsize
	
	To ensure Algorithm~\ref{alg:Yu} is an oracle for \eqref{eq:phit}, it suffices to choose the $K$ large enough so that the left hand sides of \eqref{eq:YuObjective} and \eqref{eq:YuConstraintnew} are both no more than $\hat\epsilon^2$. Because $\Lambda(D,\tilde M_0,\tilde M_1,m,\sigma_\epsilon,K,\delta)=O(\log(K/\delta))$. It suffices to choose $K=\tilde O(\frac{1}{\hat\epsilon^4}\log(\frac{1}{\delta}))$. Hence, Algorithm~\ref{alg:Yu} can be used as an oracle to solve \eqref{eq:phit} and the complexity of Algorithm~\ref{alg:iqrc} will be 
	$$
	TK=\tilde O\left(\frac{1}{\epsilon^6}\right).
	$$
	%Note that,  Algorithm~\ref{alg:MD} is deterministic so that the complexity above does not depend on $\delta$.
	
\end{proof}

%\section{Do \emph{not} have an appendix here}
%
%\textbf{\emph{Do not put content after the references.}}
%%
%Put anything that you might normally include after the references in a separate
%supplementary file.
%
%We recommend that you build supplementary material in a separate document.
%If you must create one PDF and cut it up, please be careful to use a tool that
%doesn't alter the margins, and that doesn't aggressively rewrite the PDF file.
%pdftk usually works fine. 
%
%\textbf{Please do not use Apple's preview to cut off supplementary material.} In
%previous years it has altered margins, and created headaches at the camera-ready
%stage. 
%%%%%%%%%%%%%%%%%%%%%%%%%%%%%%%%%%%%%%%%%%%%%%%%%%%%%%%%%%%%%%%%%%%%%%%%%%%%%%%
%%%%%%%%%%%%%%%%%%%%%%%%%%%%%%%%%%%%%%%%%%%%%%%%%%%%%%%%%%%%%%%%%%%%%%%%%%%%%%%

\end{document}